\documentclass[11pt,reqno]{amsart}
\usepackage{amssymb,latexsym}
\usepackage{graphicx}
\usepackage{fancyhdr,amssymb}
\usepackage{url}		
\usepackage{amsmath,amssymb}
\usepackage[mathlines]{lineno}
\usepackage{enumerate}
\usepackage{oubraces}
\usepackage{tabularx}
\usepackage{multirow}
\usepackage{tikz}
\usepackage{lmodern}
\usepackage{here}
\usepackage{cite}
\usepackage[T1]{fontenc}

\theoremstyle{plain}
\numberwithin{equation}{section}
\newtheorem{thm}{Theorem}[section]
\newtheorem{theorem}[thm]{Theorem}
\newtheorem{lemma}[thm]{Lemma}
\newtheorem{corollary}[thm]{Corollary}
\newtheorem{example}[thm]{Example}
\newtheorem{remark}[thm]{Remark}
\newtheorem{definition}[thm]{Definition}
\newtheorem{proposition}[thm]{Proposition}

\renewcommand{\epsilon}{\varepsilon}
\renewcommand{\phi}{\varphi}

\def\PV{\mathcal{P}}
\def\QV{\mathcal{Q}}
\def\PVW{\mathcal{P}_{w}}

\def\pp{\ldotp\ldotp}

\def\cd3#1{\textbf{\textsf{#1}}}
\def\sa#1{\cd3{#1}}

\def\lp{\textit{lp}}

\def\act{A}

\begin{document}


\setcounter{page}{1}

\title[Abelian Powers and Repetitions in Sturmian Words]{Abelian Powers and Repetitions in Sturmian Words}

\author[G. Fici]{Gabriele Fici}
\address[G. Fici]{Dipartimento di Matematica e Informatica\\
                Universit\`a di Palermo\\
                Italy}
\email[Corresponding author]{gabriele.fici@unipa.it}
\thanks{Some of the results in this paper have been presented without proofs at the 17th International Conference on Developments in Language Theory, DLT 2013 \cite{FiLaLeLeMiPG13}.}

\author[A. Langiu]{Alessio Langiu}
\address[A. Langiu]{ Department of Informatics\\
   King's College London\\
   UK}
\email{alessio.langiu@kcl.ac.uk}

\author[T. Lecroq]{Thierry Lecroq}
\address[T. Lecroq, A. Lefebvre, \'E. Prieur-Gaston]{LITIS EA4108, FR CNRS 3638 Normastic, Normandie Universit\'e\\ 
	Universit\'e de Rouen\\
	 France }
\email{\{thierry.lecroq,arnaud.lefebvre,elise.prieur\}@univ-rouen.fr}

\author[A. Lefebvre]{Arnaud Lefebvre}

\author[F. Mignosi]{Filippo Mignosi}
\address[F. Mignosi]{DISIM\\ Universit\`a dell'Aquila\\ 
   Italy}
\email{filippo.mignosi@di.univaq.it}

\author[J. Peltom\"{a}ki]{Jarkko Peltom\"{a}ki}
\address[J. Peltom\"{a}ki]{Turku Centre for Computer Science, TUCS\\Department of Mathematics and Statistics\\University of Turku, Finland}
\email{jspelt@utu.fi}

\author[\'E. Prieur]{\'Elise Prieur-Gaston}

\maketitle

\begin{abstract}
 Richomme, Saari and Zamboni (J.~Lond.~Math.~Soc.~83:~79--95,~2011) proved that at every position of a Sturmian word starts an  abelian power of exponent $k$ for every $k > 0$. We improve on this result by studying the maximum exponents of abelian powers and abelian repetitions (an abelian repetition is an analogue of a fractional power) in Sturmian words. We give a formula for computing the maximum exponent of an abelian power of abelian period $m$ starting at a given position in any Sturmian word of rotation angle $\alpha$. By considering all possible abelian periods $m$, we recover the result of Richomme, Saari and Zamboni.

As an analogue of the critical exponent, we introduce the abelian critical exponent $\act(s_\alpha)$ of a Sturmian word $s_\alpha$ of angle $\alpha$ as the quantity $\act(s_\alpha) = \limsup k_{m}/m=\limsup k'_{m}/m$, where $k_{m}$ (resp. $k'_{m}$) denotes the maximum exponent of an abelian power (resp.~of an abelian repetition) of abelian period $m$ (the superior limits coincide for Sturmian words). We show that $\act(s_\alpha)$ equals the Lagrange constant of the number $\alpha$. This yields a formula for computing $\act(s_\alpha)$ in terms of the partial quotients of the continued fraction expansion of $\alpha$. Using this formula, we prove that $\act(s_\alpha) \geq \sqrt{5}$ and that the equality holds for the Fibonacci word. We further prove that $\act(s_\alpha)$ is finite if and only if $\alpha$ has bounded partial quotients, that is, if and only if $s_{\alpha}$ is $\beta$-power-free for some real number $\beta$.

Concerning the infinite Fibonacci word, we prove that: \emph{i}) The longest prefix that is an abelian repetition of period $F_j$, $j>1$, has length $F_j( F_{j+1}+F_{j-1} +1)-2$ if $j$ is even or $F_j( F_{j+1}+F_{j-1} )-2$ if $j$ is odd, where $F_{j}$ is the $j$th Fibonacci number; \emph{ii}) The minimum abelian period of any factor is a Fibonacci number. 
Further, we derive a formula for the minimum abelian periods of the finite Fibonacci words: we prove that for $j\geq 3$ the Fibonacci word $f_j$, of length $F_{j}$, has minimum abelian period equal to $F_{\lfloor{j/2}\rfloor}$ if $j = 0, 1, 2\mod{4}$ or to $F_{1+\lfloor{j/2}\rfloor}$  if  $ j  = 3\mod{4}$.
\end{abstract}

\section{Introduction}

Sturmian words are infinite words having exactly $n+1$ distinct factors of each length $n\geq 0$. By the celebrated theorem of Morse and Hedlund \cite{MoHe38}, they are the aperiodic binary words with minimal factor complexity. Every Sturmian word is characterized by an irrational number $\alpha$ and a real number $\rho$ called the \emph{angle} and the \emph{initial point} respectively. The Sturmian word $s_{\alpha,\rho}$ is defined by rotating the point $\rho$ by the angle $\alpha$ in the torus $I=\mathbb{R}/\mathbb{Z}=[0,1)$ and by writing a letter $\sa{b}$ when the point falls in the interval $[0,1-\alpha)$ and a letter $\sa{a}$ when the point falls in the complement. The Fibonacci word $f=\sa{abaababaabaababaabab}\cdots$ is a well-known Sturmian word obtained by taking both the angle and the initial point equal to $\phi-1$, where $\phi=(1+\sqrt{5})/2$ is the Golden Ratio. The Fibonacci word $f$ is also the limit of the sequence of finite Fibonacci words $f_{n}$, defined by $f_{0}=\sa{b}$, $f_{1}=\sa{a}$ and $f_{j}=f_{j-1}f_{j-2}$ for every $j>1$, that are the natural counterpart of the Fibonacci numbers in the setting of words. 

Sturmian words have several equivalent definitions and a lot of combinatorial properties that make them well-studied objects in discrete mathematics and theoretical computer science. 
In fact, there exists a huge bibliography on Sturmian words (see for instance the survey papers 
\cite{BerstelRecent07,Berstel-Reutenauersurvey}, \cite[Chap. 2]{lothaire-book:2002}, 
\cite[Chap. 6]{Pitheasfogg} and the references therein). 

There are mainly two approaches to Sturmian words: one is purely combinatorial, while the other uses techniques from elementary number theory to derive correspondences between the finer arithmetic properties of the irrational $\alpha$ and the factors of the Sturmian words of angle $\alpha$. In the language of computer science, such correspondences are called \emph{semantics}. In this paper, we aim at building upon such approach by showing new semantics allowing us to give new and tight results on the abelian combinatorics of Sturmian words. Indeed, this approach extends to the abelian setting the well-known fruitful semantics that in the last decades have allowed researchers to derive deep and important results on the combinatorics of infinite words from the theory of codings of rotations and continued fractions of irrationals. Interestingly these semantics also allowed researchers to shed new light on consolidated theories by exploiting the opposite direction. A remarkable example of this is represented by the work of B. Adamczewski and Y. Bugeaud \cite{AB1,AB2}.

Concerning the maximum exponent of repetitions in Sturmian words, there exists a vast bibliography (see for example \cite{Damanik200223,BeHoZa06,Vandeth2000283,Justin2001363,Krieger200770,Be99,CaDel00} and the references therein), which stems from the seminal work on the Fibonacci word presented in \cite{MignosiPirillo}. Indeed, the study of repetitions in words is a classical subject both from the combinatorial and the algorithmic point of view. Repetitions are strictly related to the notion of periodicity. Recall that a word $w$ of length $|w|$ has a \emph{period} $p>0$ if $w_i=w_{i+p}$ for every $1\leq i \leq |w|-p$, where $w_i$ is the letter in the position $i$ of $w$. The \emph{exponent} of $w$ is the ratio $|w|/\pi_w$ between its length $|w|$ and its \emph{minimum period $\pi_w$}. When studying the degree of repetitiveness of a word, we are often interested in the factors whose exponent is at least $2$, called \emph{repetitions}. Repetitions whose exponent is an integer are called \emph{integer powers} since a word $w$ with integer exponent $k\geq 2$ can be written as $w=u^{k}$, i.e., $w$ is the concatenation of $k$ copies of a non-empty word $u$ of length $\pi_w$. If instead $k$ is not an integer, then the word $w$ is a \emph{fractional power}. In this case we can write $w=u^{\lfloor k \rfloor}u'$, where $u'$ is the prefix of $u$ of length $\pi_w(k-\lfloor k \rfloor)$. For example, the word $w=aabaaba$ is a $7/3$-power since it has minimum period $3$ and length $7$.  A good reference on periodicity is \cite[Chap.~8]{lothaire-book:2002}. 

A measure of repetitiveness of an infinite word is given by the supremum of the exponents of its factors, called the \emph{critical exponent} of the word. If this supremum $\beta$ is finite, then the word is said to be $\beta^+$-power-free (or simply $\beta$-power-free if $\beta$ is irrational, so there are no factors with exponent $\beta$). For example, the critical exponent of the Fibonacci word $f$ is $2+\phi$ \cite{MignosiPirillo}, so $f$ is $(2+\phi)$-power-free. In general, a Sturmian word $s_{\alpha,\rho}$ is $\beta$-power-free for some $\beta$ if and only if the continued fraction expansion of $\alpha$ has bounded partial quotients \cite{Mi89}. The critical exponent of $s_{\alpha,\rho}$ can be explicitly determined by a formula involving these partial quotients \cite{CaDel00,Justin2001363,Damanik200223,Pelto}.

Recently, the extension of these notions to the so-called abelian setting has received a lot of interest. 
Abelian properties of words have been studied since 
the very beginning of formal language theory and combinatorics on words. The notion of the Parikh vector of a word (see later for its definition) has become a standard and is often used without an explicit reference to the original 1966 paper by Parikh \cite{Parikh:1966:CL:321356.321364}. Abelian  powers were first considered in 1957 by Erd\H{o}s \cite{Erdos1961221} as a natural generalization of usual integer powers.
Research concerning abelian 
properties of words and languages developed afterwards in 
different directions. 
For instance, there is an increasing interest in 
abelian properties of words linked to periodicity; see, for example, 
\cite{AKP2012,CI2006,CRSZ2010,DR2012,PZ13,Richomme201179,SS2011}.
 
Recall that the Parikh vector $\PV_{w}$ of a finite word $w$ enumerates the total number of each letter of the alphabet in $w$. Therefore, two words have the same Parikh vector if and only if one can be obtained from the other by permuting letters. 
An \emph{abelian decomposition} of a word $w$ is a factorization $w=u_0u_1 \cdots u_{j-1}u_j$, where  $j\geq 2$, the words $u_1$, $\ldots$, $u_{j-1}$ have the same Parikh vector $\PV$ and the Parikh vectors of the words $u_0$ and $u_j$ are contained in $\PV$ (that is, they are component-wise less
or equal to $\PV$ but not equal to $\PV$). The sum $m$ of the components of the Parikh vector $\PV$ (that is, the length of $u_1$) is called an \emph{abelian period} of $w$ (cf.~\cite{CI2006}).  
The words $u_0$ and $u_j$, the first and the last factor of the decomposition, are respectively called the \emph{head} and the \emph{tail} of the abelian decomposition. Notice that different abelian decompositions can give the same abelian period. For example, the word $w=abab$ has an abelian period $2$ with $u_0=\epsilon$ (the empty word), $u_1=u_2=ab$ and $u_3=\epsilon$ or with $u_0=a$, $u_1=ba$ and $u_2=b$.
The \emph{abelian exponent} of $w$ is the ratio $|w|/\mu_w$ between its length $|w|$ and its minimum abelian period $\mu_w$. We say that a word $w$ is an \emph{abelian repetition} of period $m$ and exponent $k$ if it has an abelian decomposition $w=u_0u_1 \cdots u_{j-1}u_j$ with $j\geq 3$ (so that $u_1$ and $u_2$ exist and are nonempty) such that $|u_1| = m$ and $|w|/m = k$. If we are uninterested in the period and exponent, then we simply call $w$ an abelian repetition.
An \emph{abelian power} (also known as a \emph{weak repetition}~\cite{Cummings_weakrepetitions}) is a word $w$ that has an abelian decomposition with empty head and empty tail. Let $m$ be the abelian period of $w$ corresponding to the decomposition. Then we say that the word $w$ is an abelian power of period $m$ and exponent $|w|/m$. If the exponent of $w$ equals $1$, then we say that $w$ is a \emph{degenerated} abelian power of period $m$.

\subsection{Our results}

The main contribution of this paper is the description of a framework that allows us to translate arithmetic properties of rotations of an irrational number in the torus  $I=\mathbb{R}/\mathbb{Z}=[0,1)$ to properties of abelian powers and repetitions in Sturmian words.

In \cite{Mi89} (and in a very preliminary form in \cite{KnuthAMM}) a bijection (that we call \emph{Sturmian bijection}) between factors of Sturmian 
words and subintervals of the torus $I$ is described. 
In the last three decades, the Sturmian bijection has allowed researchers to shed light on the combinatorics of the factors of Sturmian words especially from the point of view of repetitions. Most of the results on maximal repetitions in Sturmian words stem in fact from the Sturmian bijection.

We show in this paper that the Sturmian bijection preserves abelian properties of
factors. In particular, using the Sturmian bijection, we prove that a Sturmian word of rotation angle $\alpha$ contains an abelian power of period $m$ and exponent $k\geq 2$ if and only if $\|m\alpha\|< \frac{1}{k}$, where $\|x\|$ is the distance between $x$ and the nearest integer. As a consequence, the maximum exponent of an abelian power of period $m$ in a Sturmian word of rotation angle $\alpha$ is $\left \lfloor 1/ \|m\alpha\|  \right \rfloor$. Furthermore, we prove that the Sturmian word  $s_{\alpha,\rho}$ of angle $\alpha$ and initial point $\rho$ contains an abelian power of period $m$ and exponent $k\geq 2$ starting at position $n$ if and only if $\{\rho+n\alpha\}<1-k\{m\alpha\}$ or $\{\rho+n\alpha\}>k\{-m\alpha\}$ save for few exceptional positions related to the points $\{\{-rm\alpha\}\colon r \geq 0\}$. We recover the result of \cite{Richomme201179} that abelian powers of arbitrarily large exponent occur at every position of a Sturmian word.

A Sturmian word always contains abelian powers of arbitrarily large exponent, so we cannot define a direct analogue of the critical exponent to the abelian setting. Instead we define the \emph{abelian critical exponent} of a Sturmian word $s_\alpha$ of angle $\alpha$ as the quantity $\act(s_\alpha) = \limsup k_{m}/m=\limsup k'_{m}/m$ where $k_m$ (resp.~$k'_m$) denotes the maximum exponent of an abelian power (resp.~of an abelian repetition) of abelian period $m$ in $s_\alpha$ (in fact, the two superior limits coincide). We show that $\act(s_\alpha)$ equals the \emph{Lagrange constant} of the irrational $\alpha$, a well-known constant in number theory (see Section \ref{sec:approx} for its definition). Via this connection, we determine $\act(s_\alpha)$ in terms of the partial quotients of the continued fraction expansion of $\alpha$. This allows us to prove that $\act(s_\alpha) \geq \sqrt{5}$ for every Sturmian word $s_\alpha$ and that the equality holds for the Fibonacci word. We further prove that $\act(s_\alpha)$ is finite if and only if the continued fraction expansion of $\alpha$ has bounded partial quotients, that is, if and only if $s_{\alpha}$ is $\beta$-power-free for some real number $\beta$.

We finally focus on the particular case of the Fibonacci word $f=s_{\phi-1,\phi-1}$. We prove that in the Fibonacci word the maximum exponent of an abelian power of period $F_j$---the $j$th Fibonacci number---equals $\lfloor \phi F_j + F_{j-1} \rfloor$ and that for every $F_j$, $j>1$, the longest prefix of the Fibonacci word that is an abelian repetition of period $F_j$ has length $F_j( F_{j+1}+F_{j-1} +1)-2$ if $j$ is even and $F_j( F_{j+1}+F_{j-1} )-2$ if $j$ is odd. 
We then prove that the minimum abelian period of any factor of the Fibonacci word is a Fibonacci number; a result analogous to a result of Currie and Saari \cite{CuSa09} concerning ordinary periods.
These results allow us to give an exact formula for the minimum abelian periods of the finite Fibonacci words. More precisely, we prove that for every $j\geq 3$ the Fibonacci word $f_j$, of length $F_{j}$, has minimum abelian period $F_{\lfloor{j/2}\rfloor}$ if $j = 0, 1, 2\mod{4}$ and $F_{1+\lfloor{j/2}\rfloor}$ if $j = 3\mod{4}$.

\medskip

The paper is organized as follows. In Section~\ref{sec:preliminaries} we give the basic definitions and fix the notation. In Section~\ref{sec:Sturmian} we recall needed results on Sturmian words and present connections with the Parikh vectors of their factors. In Section~\ref{sec:ab_pow_and_ab_repet} we give the main results about abelian powers in Sturmian words, while in Section~\ref{sec:approx} we use standard techniques from elementary number theory to study abelian repetitions in Sturmian words and the abelian critical exponent of Sturmian words. In the final Section~\ref{sec:abrepFibo} we deal with the abelian repetitions in the Fibonacci word and abelian periods of its factors.

\section{Preliminaries}\label{sec:preliminaries}

Let $\Sigma=\{a_{1},a_{2},\ldots ,a_{\sigma}\}$ be an ordered alphabet of cardinality $\sigma$, and let $\Sigma^*$ be the set of finite words over $\Sigma$. We let $|w|$ denote the length of the word $w$. The empty word has length $0$ and is denoted by $\epsilon$. We let $w_i$ denote the $i$th letter of $w$ and $w_{i\pp j}$ with $1 \leq i \leq j \leq |w|$ the factor $w_i w_{i+1} \cdots w_j$ of $w$. We say that the factor $w_{i\pp j}$ occurs at position $i$ in $w$. 
We let $|w|_a$ denote the number of occurrences of the letter $a\in\Sigma$ in the word $w$.
An integer $p>0$ is an (ordinary) period of a word $w$ if $w_i=w_{i+p}$ for all $i$ such that $1\leq i \leq |w|-p$.

The \emph{Parikh vector} of a word $w$, denoted by $\PVW$, counts the occurrences of each letter of $\Sigma$ in $w$, i.e., $\PVW=(|w|_{a_{1}},\ldots,|w|_{a_{\sigma}})$. Given the Parikh vector $\PVW$ of a word $w$, $\PVW [i]$ denotes its $i$th component
and $|\PVW|$ its norm (the sum of its components). Thus, for a word $w$ and $i$ such that $1\leq i\leq\sigma$, we have $\PVW [i]=|w|_{a_i}$ and $|\PVW|=\sum_{i=1}^{\sigma}\PVW[i]=|w|$. Given two Parikh vectors $\PV$ and $\QV$, we say that $\PV$ is \emph{contained} in $\QV$ if $|\PV| < |\QV|$ and $\PV[i]\leq \QV[i]$ for every $i$ such that $1\leq i\leq \sigma$. If the Parikh vector $\PV$ is contained in $\QV$, then we simply write $\PV \subset \QV$.

Recall from the introduction that an \emph{abelian decomposition} of a word $w$ is a factorization $w=u_0u_1 \cdots u_{j-1}u_j$ where $j \geq 2$, the words $u_1, u_2, \ldots, u_{j-1}$ have the same Parikh vector $\PV$ and the Parikh vectors of $u_0$ (the \emph{head}) and $u_j$ (the \emph{tail}) are contained in $\PV$. The norm of $\PV$ is an \emph{abelian period} of $w$.
The \emph{abelian exponent} of $w$ is the ratio $|w|/\mu_w$ between its length $|w|$ and its minimum abelian period $\mu_w$.
We say that a word $w$ is an \emph{abelian repetition} if it has an abelian decomposition $w=u_0u_1 \cdots u_{j-1}u_j$ with $j\geq 3$. 
An \emph{abelian power} is a word for which there exists an abelian decomposition with an empty head and an empty tail. An abelian power $w$ is \emph{degenerated} if $j = 2$ in its abelian decomposition $w = u_0u_1 \cdots u_{j-1}u_j$ with an empty head and an empty tail, that is, if and only if the norm of the Parikh vector of $u_1$ equals $|w|$.

\begin{example}
The word $w=abaababa$ is an abelian repetition of minimum period $2$ and abelian exponent $4$ since $w=a\cdot ba\cdot ab \cdot ab \cdot a$. Notice that $w$ is also an abelian repetition of  period $3$ and exponent $8/3$ since $w=\varepsilon \cdot aba\cdot aba \cdot ba$.

If a word $w$ is an abelian power of maximum exponent $k$, then $k$ is not necessarily the abelian exponent of $w$. For
instance, if $w = (baaba)^2$, then $w$ is an abelian power of period $5$ and exponent $2$, but its abelian exponent is
$10/3$ as $w = b \cdot aab \cdot aba \cdot aba \cdot \epsilon$.
\end{example}

The following lemma, which is a natural extension of the properties of ordinary periods to the abelian setting, is a straightforward consequence of the definition of the abelian period.

\begin{lemma}\label{lem:ap}
  Let $v$ be a factor of a word $w$. Then $\mu_w \geq \mu_v$. On the other hand, if $w$ has an abelian period $m$ such
  that $m\leq |v|$, then $m$ is also an abelian period of $v$.
\end{lemma}


\section{Sturmian Words}\label{sec:Sturmian}

From now on, we fix the alphabet $\Sigma=\{\sa{a,b}\}$ and the torus $I=\mathbb{R}/\mathbb{Z}=[0,1)$. Recall that given a real number $\alpha$,  $\lfloor \alpha \rfloor$ is the largest integer smaller than or equal to $\alpha$, $\lceil \alpha \rceil$ is the smallest integer greater than or equal to $\alpha$ and $\{\alpha\}=\alpha-\lfloor \alpha \rfloor$ is the fractional part of  $\alpha$. Notice that $\{-\alpha\}= 1-\{\alpha\}$. We let $\|\alpha\|$ denote the distance between $\alpha$ and the nearest integer, i.e., $\|\alpha\|=\min(\{\alpha\},\{-\alpha\})$. Observe that $\|\alpha\|=\|-\alpha\|$.
Most of the content present in this section is based on the results from \cite{Mi89} (see also \cite[Chap.~2]{lothaire-book:2002} and \cite[Chap.~6]{Pitheasfogg}).

Let us recall the definition of Sturmian words as codings of a rotation.
Let $\alpha\in I$ be irrational and $\rho\in I$. 
The Sturmian word $\underline{s}_{\alpha,\rho}$ (resp.~$\overline{s}_{\alpha,\rho}$) of  \emph{angle} $\alpha$ and \emph{initial point} $\rho$ is the infinite word $a_{0}a_{1}a_{2}\cdots$ defined by
$$a_{n} =
\left\{
	\begin{array}{ll}
		\sa{b}  & \mbox{if } \{ \rho + n\alpha \}\in I_{\sa{b}},\\
		\sa{a}  & \mbox{if } \{ \rho + n\alpha \}\in I_{\sa{a}}, 
	\end{array}
\right.$$ 
where $I_{\sa{b}}=[0,1-\alpha)$ and $I_{\sa{a}}=[1-\alpha,1)$   (resp.~$I_{\sa{b}}=(0,1-\alpha]$ and $I_{\sa{a}}=(1-\alpha,1]$).

In other words, take the unit circle and consider a point initially in position $\rho$. Then rotate this point on the circle (clockwise) by the angles $\alpha$, $2\alpha$, $3\alpha$, etc, and write consecutively the letters associated with the intervals the rotated points fall into. The infinite sequence of letters obtained is the Sturmian word $\underline{s}_{\alpha,\rho}$ or $\overline{s}_{\alpha,\rho}$, depending on the choice of the intervals $I_{\sa{b}}$ and $I_{\sa{a}}$. See \figurename~\ref{Fig:gab1} for an illustration. Notice that the words $\underline{s}_{\alpha,\rho}$ and $\overline{s}_{\alpha,\rho}$ differ by at most two letters: a single occurrence of $\sa{ba}$ changes into $\sa{ab}$ or vice versa. Notice that the words $\sa{ba}$ and $\sa{ab}$ are abelian equivalent.

Mostly the choice of the intervals $I_{\sa{b}}$ and $I_{\sa{a}}$ is irrelevant. We thus adopt the convention that
$s_{\alpha,\rho}$ stands for either of the Sturmian words $\underline{s}_{\alpha,\rho}$ or
$\overline{s}_{\alpha,\rho}$. When we want to emphasize which choice of intervals is used to obtain the word
$s_{\alpha,\rho}$, we write $s_{\alpha,\rho} = \underline{s}_{\alpha,\rho}$ or
$s_{\alpha,\rho} = \overline{s}_{\alpha,\rho}$. Alternatively the choice for a fixed Sturmian word is made explicit
by telling if $0 \in I_{\sa{b}}$ or $0 \notin I_{\sa{b}}$. We later focus on specific subintervals of the torus, and
the choice of the intervals $I_{\sa{b}}$ and $I_{\sa{a}}$ affects the endpoints of these subintervals. We let
$I(\alpha,\beta)$, $\alpha, \beta \in I$, $\alpha < \beta$, to stand for the subinterval $[\alpha,\beta)$ if
$0 \in I_{\sa{b}}$ and for $(\alpha,\beta]$ if $0 \notin I_{\sa{b}}$.

\begin{figure}[!ht]
\centering
\includegraphics[scale=0.3]{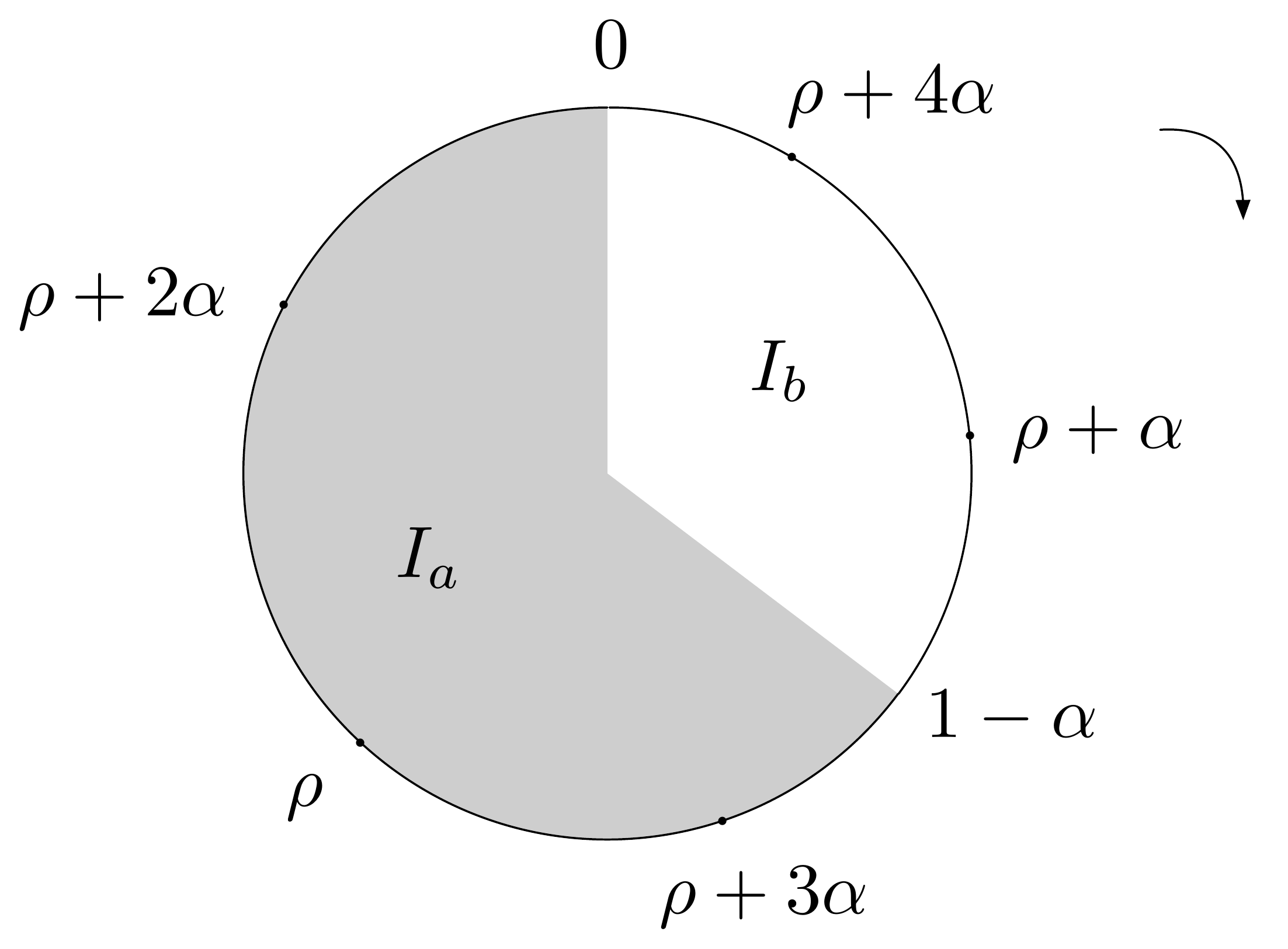} 
\caption{The rotation of the initial point $\rho = \phi-1 \approx 0.618$ by the angle $\alpha = \phi-1$ generating the Fibonacci word $f=s_{\phi-1,\phi-1}=\sa{abaababaabaabab}\cdots$. \label{Fig:gab1}}
\end{figure}

For example, let $\phi=(1+\sqrt 5 )/2\approx 1.618$ be the Golden Ratio, and consider the Sturmian word $f=s_{\phi-1,\phi-1}$, called the \emph{Fibonacci word}. Since the (approximated) first values of the sequence $(\{ (\phi-1) + n(\phi-1) \})_{n\geq 0}$ are $0.618$, $0.236$, $0.854$, $0.472$, $0.090$, $0.708$, $0.326$, $0.944$, $0.562$, $0.180$, $0.798$, $0.416$, $0.034$, and since $1-\alpha=\{-\alpha\}=2-\phi \approx0.382$, we have $$f=\sa{abaababaabaababaababaabaababaabaab}\cdots.$$ The choice of the intervals $I_{\sa{b}}$ and $I_{\sa{a}}$ is irrelevant here as none of the numbers in the sequence $(\{ (\phi-1) + n(\phi-1) \})_{n\geq 0}$ equals $\{-\alpha\}$ or $0$.

A Sturmian word for which $\rho=\alpha$, like the Fibonacci word, is called \emph{characteristic}. Notice that $\underline{s}_{\alpha,0}=\sa{b}s_{\alpha,\alpha}$ and $\overline{s}_{\alpha,0}=\sa{a}s_{\alpha,\alpha}$ for every $\alpha$.

An equivalent view is to fix the point and rotate the intervals backwards. The interval $I_{\sa{b}}=I_{\sa{b}}^{0}$ is rotated at each step, so that after $i$ rotations it is transformed into the interval $I_{\sa{b}}^{-i}=I(\{-i\alpha\},\{-(i+1)\alpha\})$, while $I_{\sa{a}}^{-i}=I\setminus I_{\sa{b}}^{-i}$. See \figurename~\ref{Fig:gab2} for an illustration.

\begin{figure}[!ht]
\centering
\includegraphics[scale=0.5]{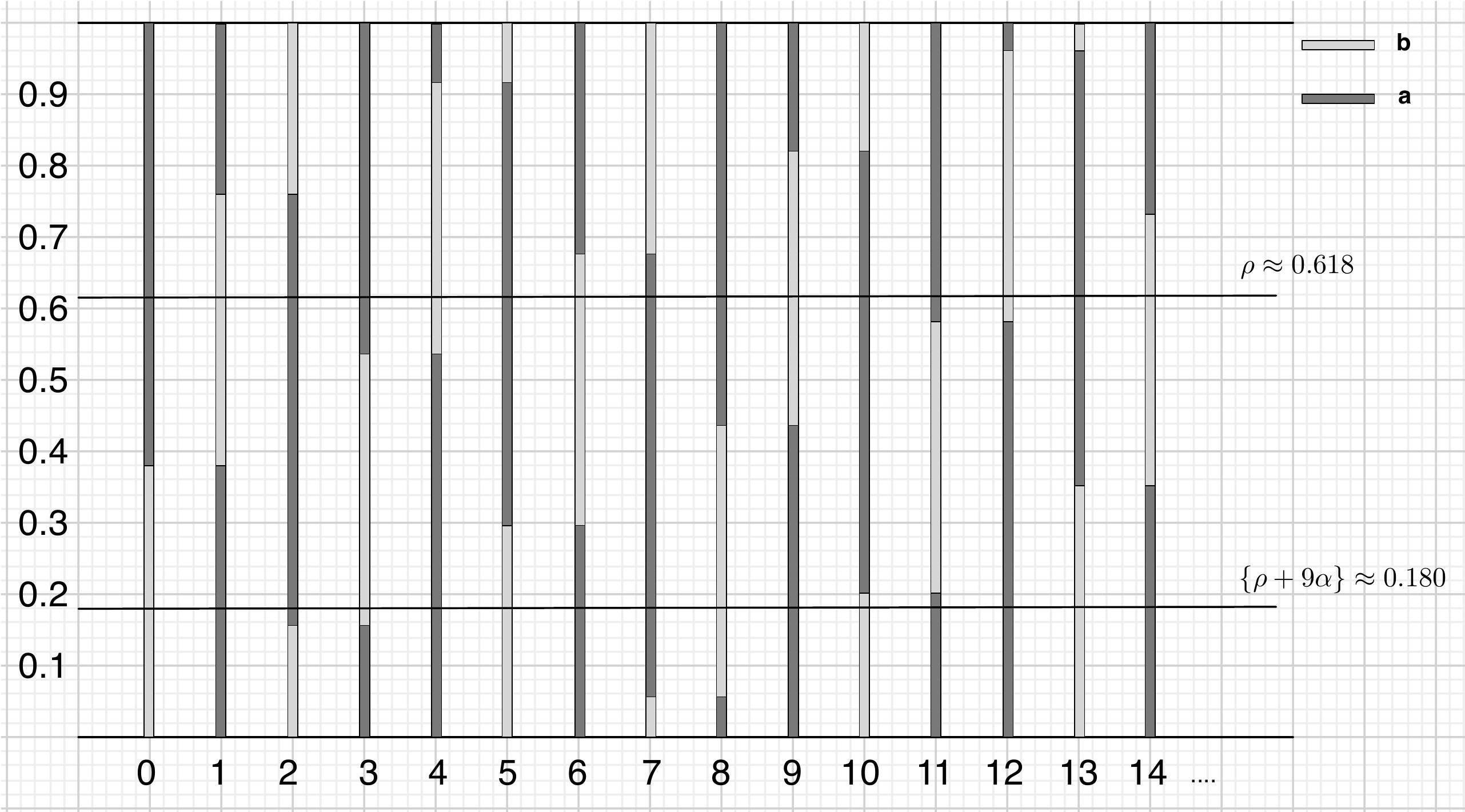} 
\caption{The interval $I_{\sa{b}}=I_{\sa{b}}^{0}=I(0,1-\alpha)$  is rotated at each step, defining the intervals $I_{\sa{b}}^{-i}=I(\{-i\alpha\},\{-(i+1)\alpha\})$ (light gray). For every $i$, the complement of the interval $I_{\sa{b}}^{-i}$ is the interval $I_{\sa{a}}^{-i}$ (dark gray). In the figure, $\alpha=\phi-1\approx 0.618$. The Fibonacci word $f$ can be obtained by looking at the horizontal line of height $\rho=\alpha$. The factor of length $15$ starting at position $9$ of the Fibonacci word $f$, that is $\sa{baababaababaaba}$, can be obtained by looking at the horizontal line of height $\{\rho+9\alpha\}$ (Proposition \ref{prima}).\label{Fig:gab2}}
\end{figure}

This representation is convenient since one can read within it not only a Sturmian word but also any of its factors. More precisely, for every positive integer $m$, the factor of length $m$ of $s_{\alpha,\rho}$ starting at position $n$ is determined only by the value of $\{\rho+n\alpha\}$, as shown in the following proposition.

\begin{proposition}\label{prima}
 Let $s_{\alpha,\rho}=a_{0}a_{1}a_{2}\cdots$ be a Sturmian word. Then, for every $n$ and $i$, we have:
 $$a_{n+i} = \left\{ \begin{array}{lllll}
\sa{b} & \mbox{if $\{\rho + n\alpha\}\in I_{\sa{b}}^{-i}$;}\\
\sa{a} & \mbox{if $\{\rho + n\alpha\}\in I_{\sa{a}}^{-i}$.}
\end{array} \right.$$
\end{proposition}

For example, suppose we want to know the factor of length 15  starting at position 9 in the Fibonacci word $f$. We have $\{\rho+9\alpha\}=\{\phi-1+9(\phi-1)\}\approx 0.180$. The first terms of the sequence $(\{-i\alpha\})_{i\geq 0}$ are, approximately, $0$, $0.382$, $0.764$, $0.146$, $0.528$, $0.910$, $0.292$, $0.674$, $0.056$, $0.438$, $0.820$, $0.202$, $0.584$, $0.966$, $0.348$, $0.729$. So we get $a_{9}a_{10}\cdots a_{23}=\sa{baababaababaaba}$ (see \figurename~\ref{Fig:gab2}).

A remarkable consequence of Proposition \ref{prima} is the following: Given a Sturmian word $s_{\alpha,\rho}$ and  a positive integer $m$, the $m+1$ different factors of $s_{\alpha,\rho}$ of length $m$ are completely determined by the intervals $I_{\sa{b}}^{0}, I_{\sa{b}}^{-1},\ldots, I_{\sa{b}}^{-(m-1)}$, that is, only by the points $\{-i\alpha\}$ for $0\leq i< m$. In particular, they do not depend on the initial point $\rho$, so the set of factors of $s_{\alpha,\rho}$ is the same as the set of factors of $s_{\alpha,\rho'}$ for any $\rho$ and $\rho'$. 
Hence, from now on, we let $s_{\alpha}$ denote any Sturmian word of angle $\alpha$.

If we arrange the $m+2$ points $0,1,\{-\alpha\},\{-2\alpha\},\ldots,\{-m \alpha\}$ in increasing order, we determine a partition of $I$ in $m+1$ half-open subintervals $L_0(m),L_{1}(m),\ldots,L_{m}(m)$. Each of these subintervals is in bijection with a factor of length $m$ of any Sturmian word of angle $\alpha$ (see \figurename~\ref{Fig:gab3}). 

\begin{figure}[!ht]
\centering
\includegraphics[scale=0.6]{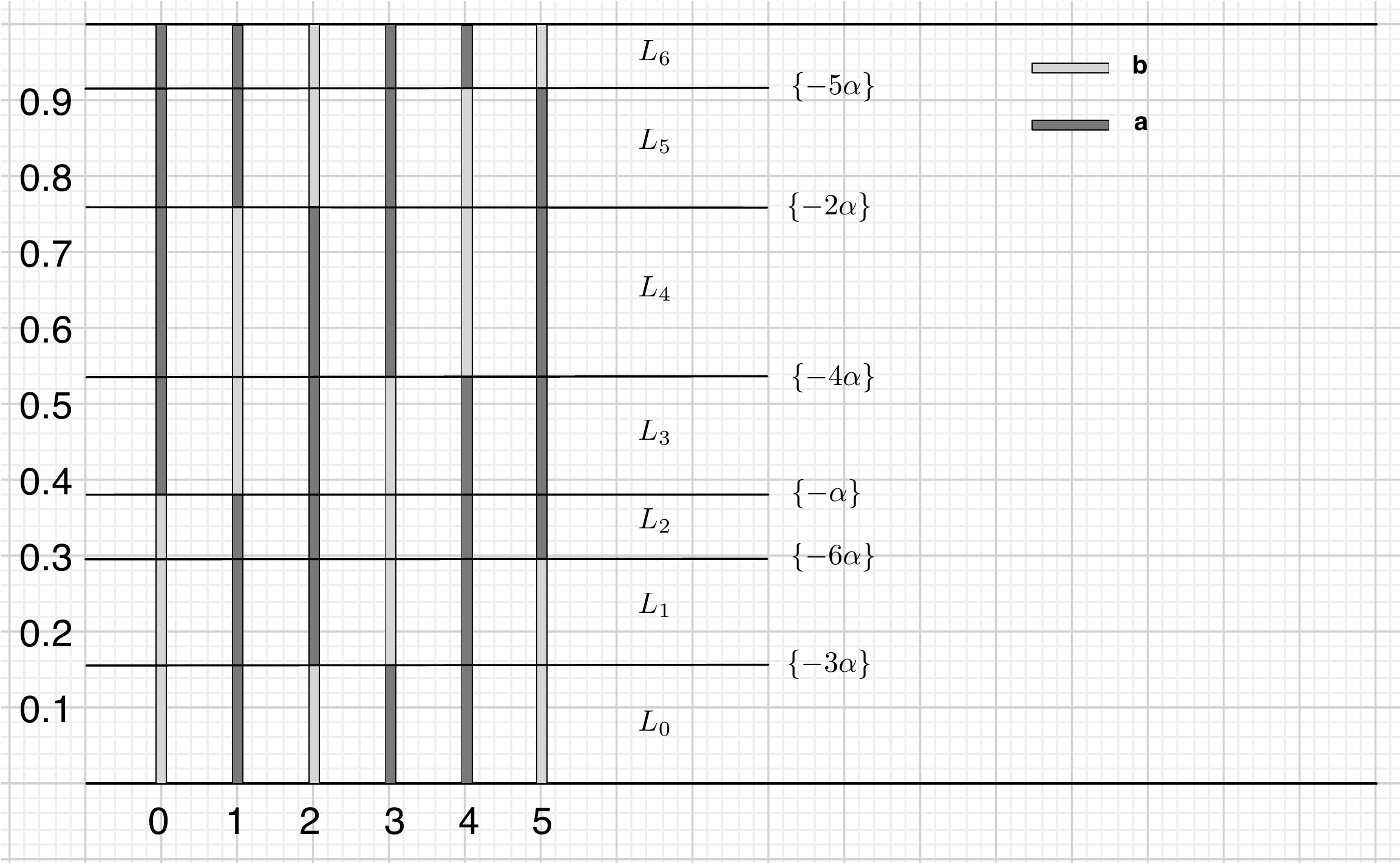} 
\caption{
The points $0$, $1$, $\{-\alpha\}$, $\{-2\alpha\}$, $\{-3\alpha\}$, $\{-4\alpha\}$, $\{-5\alpha\}$ and $\{-6\alpha\}$
arranged in increasing order. This gives the intervals $L_{0}(6)\approx I(0,0.146)$, $L_{1}(6)\approx I(0.146,0.292)$, $L_{2}(6)\approx I(0.292,0.382)$, $L_{3}(6)\approx I(0.382,0.528)$, $L_{4}(6)\approx I(0.528,0.764)$, $L_{5}(6)\approx I(0.764,0.910)$ and $L_{6}(6)\approx I(0.910,1)$. These intervals are associated respectively with the factors
$\sa{babaab}$,$\sa{baabab}$,$\sa{baabaa}$,$\sa{ababaa}$,$\sa{abaaba}$,$\sa{aababa}$ and $\sa{aabaab}$ of length $6$ of the Fibonacci word. \label{Fig:gab3}}
  
\end{figure}

Moreover, the factors associated with these intervals are lexicographically ordered (decreasingly), as stated in the following proposition. 

\begin{proposition}\label{Pro:otherbranch}
Let $m\geq 1$ and $0\le j,k\le m$. Then  the factor 
associated with the interval  $L_j(m)$ is lexicographically greater than the factor 
associated with the interval $L_k(m)$ if and only if $j<k$.
\end{proposition}

\begin{proof}
  We prove the statement by induction on $m$. The case $m=1$ is true by the definition of the two subintervals $L_{0}(1)=I_{\sa{b}}$ and $L_{1}(1)=I_{\sa{a}}$. 
Suppose now that the statement holds for some $m$ such that $m\ge 1$ and let us show that it holds for $m+1$. 
The sequence of subintervals  $L_i(m)$  is the same sequence of intervals $L_{i}(m+1)$ except for the interval $L_{t}(m)$ containing the point $\{-(m+1)\alpha\}$, which is subdivided into two intervals.
Therefore, each of the $m+1$ factors of length $m$ is extended to the right by one letter in a unique way, except for the factor $t$ associated with the interval $L_{t}(m)$, that is extended to $t\sa{b}$, associated with the interval $L_{t}(m+1)$, and to $t\sa{a}$, associated with the interval $L_{t+1}(m+1)$. Thus, for these two factors the statement holds.
For the factors of length $m+1$ different from $t\sa{b}$ and $t\sa{a}$ we can apply the induction hypothesis.
\end{proof}

The result of Proposition \ref{Pro:otherbranch} is of independent interest and is related to some recent research on Sturmian words and the lexicographic order (see \cite{Bucci201225,Glen200845,JeZa04,Perrin2012265}).

We now present properties of the bijection between the factors of length $m$ and the intervals $L_{i}(m)$ that allow us to use standard Number Theory techniques to deal with abelian repetitions in Sturmian words and, in particular, in the Fibonacci word.

Recall that a factor of length $m$ of a Sturmian word  $s_{\alpha}$ has a Parikh vector equal either to $(\lfloor m\alpha \rfloor , m-\lfloor m\alpha \rfloor )$ (in which case it is called \emph{light}) or to $(\lceil m\alpha \rceil , m-\lceil m\alpha \rceil) $ (in which case it is called \emph{heavy}).
The following proposition relates the intervals $L_{i}(m)$ to the Parikh vectors of the associated factors; this result appears as a part of \cite[Theorem~19]{Rigo13}.

\begin{proposition}\label{pro:main}
Let $s_{\alpha}$ be a Sturmian word of angle $\alpha$ and $m$ be a positive integer.
Let $t_{i}$ be the factor of length $m$ associated with the interval $L_{i}(m)$. Then $t_{i}$ is
  heavy if  $L_{i}(m)\subset I(\{-m\alpha\},1)$, while it is light if  $L_{i}(m)\subset I(0,\{-m\alpha\})$.

Moreover, if $\{-m\alpha\}\geq \{-\alpha\}$, then all heavy factors start and end with $\sa{a}$, while if $\{-m\alpha\} \leq \{-\alpha\}$, then all light factors start and end with $\sa{b}$.
\end{proposition}

\begin{proof}
We prove the statement by induction on $m$. The case $m=1$ is true by definition. Suppose the
statement true for $m$, and let us prove it for $m+1$. We have two cases: 
\begin{enumerate}
 \item $\{-(m+1)\alpha\} < \{-m\alpha\}$; 
 \item $\{-(m+1)\alpha\} > \{-m\alpha\}$.
\end{enumerate}

Case 1. By induction, the factors of length $m$ corresponding to intervals above $\{-m\alpha\}$ are heavy and the others are light. Hence, the factors of length $m+1$ corresponding to intervals above $\{-(m+1)\alpha\}$ are either heavy factors of length $m$ extended with $\sa{b}$ or light factors of length $m$ extended with $\sa{a}$, while the factors of length $m+1$ corresponding to intervals below $\{-(m+1)\alpha\}$ are light factors of length $m$ extended with $\sa{b}$. Therefore, the former are the heavy factors of length $m+1$, while the latter are the light ones.

Case 2. In this case, the factors of length $m+1$ corresponding to intervals above $\{-(m+1)\alpha\}$ are heavy factors of length $m$ extended with $\sa{a}$, so they must be heavy factors of length $m+1$. The factors of length $m+1$ corresponding to intervals below $\{-(m+1)\alpha\}$ have another Parikh vector, so they are the light ones.

The second part of the statement follows directly from the very definition of Sturmian words as rotation words. 
\end{proof}

\begin{example}
 Let $\alpha=\phi-1\approx 0.618$ and $m=6$. We have $6\alpha\approx 3.708$, so $\{-6\alpha\}\approx 0.292$. It is evident from \figurename~\ref{Fig:gab3} that the factors of length $6$ 
corresponding to intervals above (resp.~below) $\{-6\alpha\}\approx 0.292$ all have Parikh vector $(4,2)$ (resp.~$(3,3)$). That is, the intervals $L_{0}$ and $L_{1}$ are associated with the light factors (\sa{babaab}, \sa{baabab}), while the intervals from $L_{2}$ to $L_{6}$ are associated with the heavy factors (\sa{baabaa}, \sa{ababaa}, \sa{abaaba}, \sa{aababa}, \sa{aabaab}). Notice that every light factor starts and ends with $\sa{b}$ since $\{-6\alpha\}<\{-\alpha\}$.
\end{example}

From Propositions \ref{prima} and \ref{pro:main}, we derive the following.

\begin{corollary}\label{cor:hl}
Let $s_{\alpha,\rho}$ be a Sturmian word. Then for all integers $m$ and $n$, the factor of length $m$ occurring in $s_{\alpha,\rho}$ at position $n$ is heavy if $\{\rho+n\alpha\}>\{-m\alpha\}$, while it is light if $\{\rho+n\alpha\}<\{-m\alpha\}$.
\end{corollary}

\section{Abelian Powers in Sturmian Words}\label{sec:ab_pow_and_ab_repet}

The results of the previous section can be used to give tight bounds on the lengths of abelian powers in Sturmian words.

The next results extend to the abelian setting analogous results obtained for ordinary powers in \cite{Mi89}.
First, observe that an abelian power in a Sturmian word is a concatenation of factors of the same length having equal Parikh vectors, that is, these factors are all heavy or all light.

The next proposition follows directly from Propositions \ref{prima} and \ref{pro:main}.

\begin{proposition}\label{pro:int}
Let $s_{\alpha,\rho}=a_{0}a_{1}a_{2}\cdots$ be a Sturmian word of angle $\alpha$. Then the factor $a_{n}\cdots a_{n+m-1}\cdots a_{n+km-1}$ is an abelian power of period $m$ and exponent
$k\geq 2$ starting at position $n$ if and only if the $k$ points $\{\rho+(n+im)\alpha\}$, $0\leq i \leq k-1$, are all either in the interval $I(0,\{-m\alpha\})$ or in the interval $I(\{-m\alpha\},1)$.
\end{proposition}

In fact, we can state something more precise. 

\begin{lemma}\label{lem:ordered}
The $k$ points $\{\rho+(n+im)\alpha\}$, $0\leq i \leq k-1$, of Proposition \ref{pro:int} are naturally ordered. That is to say: 
\begin{itemize}
 \item if $\{m\alpha\} < 1/2$, then they are all in the subinterval 
$I(0, \{-m\alpha\})$ and $$\{\rho+n\alpha\}<\{\rho+(n+m)\alpha\}< \cdots < \{\rho+(n+(k-1)m)\alpha\};$$ 
\item if  $\{m\alpha\}> 1/2$, then they are all in the interval $I(\{-m\alpha\},1)$ 
and $$\{\rho+n\alpha\} > \{\rho+(n+m)\alpha\}> \cdots > \{\rho+(n+(k-1)m)\alpha\}.$$
\end{itemize}
\end{lemma}

\begin{proof}
We prove only the first part; the latter part is similar. Recall that $k \geq 2$. If
$\{\rho+n\alpha\} \in I(\{-m\alpha\},1)$, then $\{\rho+(n+m)\alpha\} \notin I(0,\{-m\alpha\})$ because
$\{m\alpha\} < 1/2$. Thus, by Proposition \ref{pro:int} we conclude that $\{\rho+n\alpha\} \in I(0,\{-m\alpha\})$.
Therefore, by assumption the points
$\{\rho+(n+im)\alpha\}$, $0\leq i \leq k-1$, are in $I(0,\{-m\alpha\})$. Let then $i < k-1$. Since
$\{\rho+(n+im)\alpha\} < \{-m\alpha\}$, it follows that $\{\rho+(n+im)\alpha\} + \{m\alpha\} < 1$, so
$\{\rho+(n+(i+1)m)\alpha\} = \{\rho+(n+im)\alpha\} + \{m\alpha\} > \{\rho+(n+im)\alpha\}$. The conclusion follows.
\end{proof}

\begin{example}
 Let $f=a_{0}a_{1}a_{2}\cdots$ be the Fibonacci word. The factor
 \[
 a_{9}a_{10}\cdots a_{23}=\sa{baababaababaaba}
 \]
 is an abelian power of period $m=5$ and exponent $k=3$ (it is also an ordinary power, but this is irrelevant here). The sequence $\{\rho+9\alpha\}$, $\{\rho+14\alpha\}$, $\{\rho+19\alpha\}$, i.e., the sequence $\{10(\phi-1)\}$, $\{15(\phi-1)\}$, $\{20(\phi-1)\}$, is (approximately) equal to the sequence $0.180$, $0.271$, $0.361$. This sequence is increasing---agreeing with Lemma \ref{lem:ordered}, since $\{m\alpha\}=\{5(\phi-1)\}\approx 0.090<0.5$---and is contained in the interval $I(0,\{-m\alpha\})\approx I(0,0.910)$.
\end{example}

\begin{remark}
 Notice that all the Sturmian words with the same rotation angle $\alpha$ have the same abelian powers and the same abelian repetitions---of course starting at different positions, depending on the value of the initial point $\rho$.
\end{remark}

In the following theorem we characterize the positions of occurrence of the abelian powers having given period and exponent.
For most positions the case \emph{(i)} of the next theorem applies, but due to the choice involved in coding the points $0$
and $1-\alpha$, the special points $\{\{-rm\alpha\}\colon r \geq 0\}$ require specific attention.

\begin{theorem}\label{the:maingab}
  Let $s_{\alpha,\rho} = a_0 a_1 a_2 \cdots$ be a Sturmian word of angle $\alpha$. Consider the factor
  $w = a_n \cdots a_{n+m-1} \cdots a_{n+km-1}$ starting at position $n$ of $s_{\alpha,\rho}$.
  \begin{enumerate}[(i)]
    \item If $\{\rho + n\alpha\} \notin \{\{-rm\alpha\}\colon r \geq 0\}$, then the factor $w$ is an abelian power of
          period $m$ and exponent $k \geq 2$ if and only if $\{\rho+n\alpha\} < 1-k\{m\alpha\}$ (if
          $\{m\alpha\} < 1/2$) or $\{\rho+n\alpha\} > k\{-m\alpha\}$ (if $\{m\alpha\} > 1/2$).
    \item If $\{\rho + n\alpha\} = 0$, then the factor $w$ is an abelian power of period $m$ and exponent $k \geq 2$
          if and only if $0 \in I_{\sa{b}}$ and $k\{m\alpha\} < 1$ (if $\{m\alpha\} < 1/2$) or $0 \notin I_{\sa{b}}$
          and $k\{-m\alpha\} < 1$ (if $\{m\alpha\} > 1/2$).
    \item If $\{\rho + n\alpha\} = \{-rm\alpha\}$ for some $r > 0$, then the factor $w$ is an abelian power of period
          $m$ and exponent $k$ such that $2 \leq k < r$ if and only if $\{\rho+n\alpha\} < 1-k\{m\alpha\}$ (if
          $\{m\alpha\} < 1/2$) or $\{\rho+n\alpha\} > k\{-m\alpha\}$ (if $\{m\alpha\} > 1/2$).
    \item If $\{\rho + n\alpha\} = \{-rm\alpha\}$ for some $r > 0$, then the factor $w$ is an abelian power of period
          $m$ and exponent $k$ such that $k \geq r \geq 2$ if and only if $0 \notin I_{\sa{b}}$ and
          $\{\rho+n\alpha\} < 1-k\{m\alpha\}$ (if $\{m\alpha\} < 1/2$) or $0 \in I_{\sa{b}}$ and
          $\{\rho+n\alpha\} > k\{-m\alpha\}$ (if $\{m\alpha\} > 1/2$).
  \end{enumerate}
\end{theorem}
\begin{proof}
  \emph{(i)} Suppose that $\{\rho + n\alpha\} \notin \{\{-rm\alpha\}\colon r \geq 0\}$ and $\{m\alpha\} < 1/2$ (the
  case $\{m\alpha\} > 1/2$ is analogous). Say the factor $w$ is an abelian power of period $m$ and exponent
  $k \geq 2$. Since $k \geq 2$ all of the points $\{\rho + \{n + im\}\alpha\}$, $0 \leq i \leq k - 1$, are by Lemma
  \ref{lem:ordered} naturally ordered in the interval $I(0,\{-m\alpha\})$. Moreover, these points are all interior
  points of the interval $I(0,\{-m\alpha\})$, so the coding is unambiguous. The distance between any two consecutive
  such points is $\{m\alpha\}$. Therefore, $\{\rho + n\alpha\} + (k-1)\{m\alpha\}$ must be smaller than the length of
  the interval $I(0,\{-m\alpha\})$, which is equal to $\{-m\alpha\} = 1 - \{m\alpha\}$. From this we derive that
  $\{\rho+n\alpha\} < 1 - k\{m\alpha\}$.

  Conversely if $\{\rho+n\alpha\} < 1 - k\{m\alpha\}$ for $k \geq 2$, then surely the points
  $\{\rho + \{n + im\}\alpha\}$, $0 \leq i \leq k - 1$, all are interior points of the interval $I(0,\{-m\alpha\})$, so
  $w$ is indeed an abelian power of period $m$ and exponent $k$ by Proposition \ref{pro:int}.

  \emph{(ii)} Assume that $\{\rho + n\alpha\} = 0$ and $\{m\alpha\} < 1/2$ (the case $\{m\alpha\} > 1/2$ is analogous).
  Suppose that the factor $w$ is an abelian power of period $m$ and exponent $k \geq 2$. Like above in the case
  \emph{(i)}, all of the points $\{\rho + \{n + im\}\alpha\}$, $0 \leq i \leq k - 1$, are naturally ordered in
  the interval $I(0,\{-m\alpha\})$. Therefore, $0 \in I_{\sa{b}}$. Proceeding as above, we see that
  $(k-1)\{m\alpha\} < \{-m\alpha\}$, that is, $k\{m\alpha\} < 1$. The converse is easily seen to hold.

  \emph{(iii)} This case reduces directly to the case \emph{(i)} as none of the points $\{\rho + \{n+im\}\alpha\}$,
  $0 \leq i \leq k - 1$, equal neither of the two problematic points $0$ and $1-\alpha$ whose codings depend on the
  choice of the intervals $I_{\sa{a}}$ and $I_{\sa{b}}$.

  \emph{(iv)} Assume that $\{\rho + n\alpha\} = \{-rm\alpha\}$ for some $r > 0$. Suppose moreover that
  $\{m\alpha\} < 1/2$; the case $\{m\alpha\} > 1/2$ is analogous. Assume first that the factor $w$ is an abelian power
  of period $m$ and exponent $k$ such that $k \geq r \geq 2$. Again, the points $\{\rho + \{n + im\}\alpha\}$,
  $0 \leq i \leq k - 1$, are naturally ordered in the interval $I(0,\{-m\alpha\})$. Thus,
  $\{\rho+(n+(r-1)m)\alpha\} = \{-m\alpha\} \in I(0,\{-m\alpha\})$, that is to say, $0 \notin I_{\sa{b}}$. Proceeding
  exactly as in the case \emph{(i)}, we see that $\{\rho+n\alpha\} < 1-k\{m\alpha\}$.

  Conversely if $0 \notin I_{\sa{b}}$ and $\{\rho+n\alpha\} < 1-k\{m\alpha\}$ with $k \geq r \geq 2$, then again
  $w$ is an abelian power of period $m$ and exponent $k$ by Proposition \ref{pro:int}.
\end{proof}

\begin{example}
 An abelian power of period $2$ and exponent $4$ occurs in the Fibonacci word at every position $n$ such that $\{(n+1)(\phi-1)\}<1-4\{2(\phi-1)\}\approx 0.056$. The first such $n$ are $12$, $33$, $46$, $67$ and $88$.

 An abelian power of period $3$ and exponent $6$ occurs in the Fibonacci word at every position $n$ such that $\{(n+1)(\phi-1)\}>6\{-3(\phi-1)\}\approx 0.875$. The first such $n$ are $7$, $15$, $20$, $28$, $41$, $49$, $54$, $62$ and $70$.
\end{example}

Theorem \ref{the:maingab} allows us to effortlessly characterize the maximum exponent of an abelian power of period $m$.

\begin{theorem}\label{the:main1} 
  Let $s_\alpha$ be a Sturmian word of angle $\alpha$ and $m$ be a positive integer. Then $s_\alpha$ contains an
  abelian power of period $m$ and exponent $k \geq 2$ if and only if $\|m\alpha\| < \frac{1}{k}$. In particular,
  the maximum exponent $k_m$ of an abelian power of period $m$ in $s_\alpha$ is the largest integer $k$ such that
  $\|m\alpha\|< \frac{1}{k}$, i.e.,
  \begin{equation*}
    k_{m}=\left \lfloor \frac{1}{ \|m\alpha\| } \right \rfloor.
  \end{equation*}
\end{theorem}
\begin{proof}
  Keeping the period $m$ fixed, it is evident from Theorem \ref{the:maingab} that in order to maximize the
  exponent, we can consider the prefixes of the Sturmian words $\underline{s}_{\alpha,0}$ and
  $\overline{s}_{\alpha,0}$. If $\{m\alpha\} < 1/2$, then the word
  $\underline{s}_{\alpha,0} = \sa{b}s_{\alpha,\alpha}$ has an abelian power of period $m$ and maximum exponent
  $\left\lfloor 1/\|m\alpha\| \right\rfloor$ as a prefix, while if $\{m\alpha\} > 1/2$, then the word
  $\overline{s}_{\alpha,0} = \sa{a}s_{\alpha,\alpha}$ starts with an abelian power of period $m$ and maximum exponent
  $\left\lfloor 1/\|m\alpha\| \right\rfloor$.
\end{proof}

\begin{example}
In Table \ref{tab:fici2} we give the first values of the sequence $k_{m}$ for the Fibonacci word $f$. We have $k_{2}=4$, since $\{2(\phi-1)\}\approx 0.236$, so the largest $k$ such that $\{2(\phi-1)\}< 1/k$ is $4$. Indeed, $\sa{babaabab}$ is an abelian power of period $2$ and exponent $4$, and the reader can verify that no factor of $f$ of length $10$ is an abelian power of period $2$.

For $m=3$, since $\{-3(\phi-1)\}\approx 0.146$, the largest $k$ such that $\{-3(\phi-1)\}< 1/k$ is $6$. Indeed, $\sa{aabaababaabaababaa}$ is an abelian power of period $3$ and exponent $6$, and the reader can verify that no factor of $f$ of length $21$ is an abelian power of period $3$. 
\end{example}

\begin{table}
\centering  
\begin{small}
\begin{raggedright}
\begin{tabular}{c *{30}{@{\hspace{3.1mm}}c}}
$m$\hspace{2mm} & \textbf{1} & \textbf{2} & \textbf{3} & 4 & \textbf{5} & 6 & 7 & \textbf{8} & 9 & 10 & 11 & 12 & \textbf{13} & 14 & 15 & 16 & 17 & 18 & 19 & 20 & \textbf{21}
\\
\hline \\
$k_{m}$\hspace{2mm} & \textbf{2} & \textbf{4} & \textbf{6} & 2 & \textbf{11} & 3 & 3 & \textbf{17} & 2 & 5 & 4 & 2 & \textbf{29} & 2 & 3 & 8 & 2 & 8 & 3 & 2 & \textbf{46}
\\
\hline \rule[0pt]{0pt}{12pt}
\end{tabular}
\end{raggedright}\caption{\label{tab:fici2} The first values of the maximum exponent $k_{m}$ of an abelian power of period $m$ in the Fibonacci word $f$. The values corresponding to the Fibonacci numbers are in bold.}
\end{small}
\end{table}

Next we consider the maximum exponent of an abelian power of given period and location. Again, save for the
exceptional points $\{\{-rm\alpha\}\colon r \geq 0\}$, the first formula of the next corollary suffices.

\begin{corollary}\label{cor:gab2}
  Let $s_{\alpha,\rho}=a_0 a_1 a_2 \cdots$ be a Sturmian word of angle $\alpha$,
  \begin{equation*}
    A = \left\lfloor \frac{\{-\rho-n\alpha\}}{\{m\alpha\}} \right\rfloor \ \text{ and } \ 
    B = \left\lfloor \frac{\{\rho+n\alpha\}}{\{-m\alpha\}} \right\rfloor.
  \end{equation*}
  Consider the maximum exponent $k_{m,n}$ of a (possibly degenerated) abelian
  power of period $m$ starting at position $n$ in $s_{\alpha,\rho}$.
  \begin{enumerate}[(i)]
    \item If $\{\rho + n\alpha\} \notin \{\{-rm\alpha\}\colon r \geq 0\}$, then
          \begin{equation*}
            k_{m,n} = \max\left( A,B \right).
          \end{equation*}
    \item If $\{\rho + n\alpha\} = 0$, then
          \begin{equation*}
            k_{m,n} = \begin{cases}
                        \left\lfloor 1/\{m\alpha\} \right\rfloor,  &\text{ if $0 \in I_{\sa{b}}$,} \\
                        \left\lfloor 1/\{-m\alpha\} \right\rfloor, &\text{ if $0 \notin I_{\sa{b}}$.}
                      \end{cases}
          \end{equation*}
    \item If $\{\rho + n\alpha\} = \{-rm\alpha\}$ for some $r > 0$ and $r > \max(A,B)$, then
          \begin{equation*}
            k_{m,n} = \max\left( A,B \right).
          \end{equation*}
    \item If $\{\rho + n\alpha\} = \{-rm\alpha\}$ for some $r > 0$ and $r \leq \max(A,B)$, then
          \begin{equation*}
            k_{m,n} = \max\left( A - \gamma, B + \gamma - 1 \right),
          \end{equation*}
          where
          \begin{equation*}
            \gamma = \begin{cases}
                       1, &\text{ if $0 \in I_{\sa{b}}$,} \\
                       0, &\text{ if $0 \notin I_{\sa{b}}$.}
                     \end{cases}
          \end{equation*}
  \end{enumerate}
\end{corollary}
\begin{proof}
  The formulas follow directly from Theorem \ref{the:maingab}. Observe that if $\{\rho + n\alpha\} \neq 0$ and
  $\{\rho + n\alpha\} \neq \{-m\alpha\}$, then $A \geq 1$ if and only if $B = 0$. We show here how the case \emph{(iv)}
  is handled.

  Suppose that $\{\rho + n\alpha\} = \{-rm\alpha\}$ for some $r > 0$ and $r \leq \max(A,B)$. Assume that
  $\{m\alpha\} < 1/2$. By Theorem \ref{the:maingab} \emph{(iv)} there is an abelian power of period $m$ and maximum
  exponent $A \geq 2$ starting at position $n$ of $s_{\alpha,\rho}$ provided that $0 \notin I_{\sa{b}}$. If
  $0 \in I_{\sa{b}}$, then there is an abelian power of period $m$ and maximum exponent $A-1$ starting at
  position $n$ because the change of coding affects the Parikh vector of the factor of length $m$ starting at position
  $n+(A-1)m$. Therefore, $k_{m,n} = A - \gamma$ if $A > 1$. Notice that in this case $B + 1 - \gamma \leq 1$. If $A = 1$,
  then $r = 1$ by assumption, so $B = 1$. Since $A = 1$, the Parikh vectors of the factors of length $m$ starting at
  positions $n$ and $n+m$ are different when $0 \notin I_{\sa{b}}$. Since $\{\rho+(n+m)\alpha\} = 0$, the Parikh vectors
  of the factors of length $m$ starting at positions $n$ and $n+m$ are also different when $0 \in I_{\sa{b}}$. Therefore,
  $k_{m,n} = 1 = \max\left( A - \gamma, B + \gamma - 1 \right)$. The case $\{m\alpha\} > 1/2$ is similar.
\end{proof}

Hence, given a Sturmian word $s_{\alpha,\rho}$ we can compute the maximum length of an abelian power of any period $m$ starting at any position $n$ in $s_{\alpha,\rho}$. This length is precisely $m\cdot k_{m,n}$.

Corollary \ref{cor:gab2} implies the following result of Richomme, Saari and Zamboni \cite{Richomme201179}.

\begin{proposition}
  Let $s_\alpha$ be a Sturmian word of angle $\alpha$. For all $n \geq 0$ and $k \geq 1$ there is an abelian $k$-power
  starting at position $n$ of $s_\alpha$.
\end{proposition}
\begin{proof}
  By the well-known Kronecker Approximation Theorem (see, for instance, \cite[Chap.~XXIII]{Hardy_and_Wright}) we have that the sequence $(\|m\alpha\|)_{m \geq 0}$ is dense in $I$, so that we can make the quantity $\|m\alpha\|$ arbitrarily small. The claim follows
  then from Corollary \ref{cor:gab2}.
\end{proof}

\begin{example}
 For the Fibonacci word, the first values of the sequences $k_{3,n}$ and $k_{10,n}$ are given in Table \ref{tab:kmn}. For $n=0$ we have $k_{3,0}=4$, so the longest abelian power of period $3$ starting at position $0$ has exponent $4$; for $n=1$ we have $k_{3,1}=1$, so there are no proper abelian powers of period $3$ starting at position $1$; for $n=2$ we have $k_{3,2}=5$, so the longest abelian power of period $3$ starting at position $2$ has exponent $5$; etc.
\end{example}

\begin{table}
\centering  
\begin{small}
\begin{raggedright}
\begin{tabular}{c *{30}{@{\hspace{3.1mm}}c}}
$n$\hspace{2mm} &0 & 1 & 2 & 3 & 4 & 5 & 6 & 7 & 8 & 9 & 10 & 11 & 12 & 13 & 14 & 15 & 16 & 17 & 18 & 19 & 20
\\
\hline \\
$k_{3,n}$\hspace{2mm} &4& 1& 5& 3& 1& 4& 2& 6& 3& 1& 5& 2& 1& 4& 1& 6& 3& 1& 5& 2& 6
\\
\hline \\
$k_{10,n}$\hspace{2mm} &2& 4& 1& 2& 5& 1& 3& 1& 2& 4& 1& 3& 5& 1& 4& 1& 2& 4& 1& 3& 1
\\
\hline \rule[0pt]{0pt}{12pt}
\end{tabular}
\end{raggedright}\caption{\label{tab:kmn} The first values of the maximum exponent $k_{3,n}$ of a (possibly degenerated) abelian power of period $3$  starting at position $n$ and $k_{10,n}$ of a (possibly degenerated) abelian power of period $10$ starting at position $n$ in the Fibonacci word $f=s_{\phi-1,\phi-1}$.}
\end{small}
\end{table}

We now introduce a new notion, the \emph{guaranteed exponent with anticipation $i$}, which will be useful when we study the Fibonacci word in the final section. Recall from the previous section the definition of the ordered $i+1$ subintervals, $L_0(i),L_{1}(i),\ldots,L_{i}(i)$, which form a partition of the torus $I$ and are in one-to-one correspondence with the factors of length $i$.

\begin{definition}
Let $s_{\alpha}$ be a Sturmian word of angle $\alpha$. We define for all $m>0$ and $i$ such that $0\leq i\leq m$ the guaranteed exponent with anticipation $i$, denoted by $k_m^{(i)}$, as the largest $k$ such that for every $n\geq 0$ there exists a position $j$ such that $0\leq j \leq i$ and there is in $s_{\alpha}$ a (possibly degenerated) abelian power of period $m$ and exponent $k$ starting at position $n-j$.
\end{definition}

In other words, $k_m^{(i)}$ is the largest value that is guaranteed to appear in every interval of $i+1$ consecutive positions in the sequence $k_{m,n}$.

\begin{theorem}\label{theor:kmi}
 For every $m>0$ and $i$ such that $0\leq i\leq m$, we have that
 \begin{equation}\label{eq:maxk2}
k_m^{(i)}=\max \left( 1,\left \lfloor \frac{1-l_{i}}{\|m\alpha\|}\right \rfloor \right),
\end{equation}
where $l_{i}=\max_{0\leq k \leq i} |L_{k}(i)| $ is the maximum size of an interval in the Sturmian bijection with the factors of length $i$. 
\end{theorem}

\begin{proof}
Let $m > 0$ be fixed, and let us first consider the case $i = 0$. Since $l_0 = 1$, we need to prove that
$k_m^{(0)} = 1$. It is equivalent to say that in a Sturmian word $s_{\alpha,\rho}$ there always exists a position $n$
such that no proper abelian power of period $m$ starts at this position. This is clear: if $\{m\alpha\} < 1/2$, then by
Proposition \ref{pro:int} and Lemma \ref{lem:ordered} we need to find a point $\{\rho+n\alpha\}$ such that
$\{\rho+n\alpha\} > \{-m\alpha\}$, while if $\{m\alpha\} > 1/2$, we need to have $\{\rho+n\alpha\} < \{-m\alpha\}$. By
the Kronecker Approximation Theorem such a point can always be found.

Let us now consider the general case $i>0$.
We know from Proposition \ref{pro:int} that  the factor $a_{n}\cdots a_{n+m-1}\cdots a_{n+km-1}$ of length $km$ of $s_{\alpha,\rho}$ is an abelian power of period $m$ and exponent $k$ starting at position $n$ if and only if  the $k$ points $\{\rho+(n+tm)\alpha\}$, $0\leq t \leq k-1$, are either all in the interval $I_1=I(0,\{-m\alpha\})$ or all in the interval $I_2=I(\{-m\alpha\},1)$.

Let us suppose $\{m\alpha\}<1/2$; the case $\{m\alpha\}>1/2$ is analogous. The longest abelian power of period $m$ starting at position $n$ is a factor that depends on the point $\{\rho+n\alpha\}$.  Since we want the largest abelian power of period $m$ with anticipation $i$, we have to consider all the  points $\{\rho+(n-j)\alpha\}$  with $0\leq j \leq i$. Let $k$ be such that $\{\rho+n\alpha\}\in L_k(i)$.
Since the size of the interval $L_k(i)$ is at most $l_i$, by the definition of the intervals we can say that there exists a $j$, with $0\leq j \leq i$, such that $\{\rho+(n-j)\alpha\}<l_i$, but by the Kronecker Approximation Theorem, this point can be arbitrarily close to $l_i$.
Hence, the largest integer $k$ such that for any $n\geq 0$ there exists a $j\leq i$ such that $\{\rho+(n-j+tm)\alpha\} \leq \{-m\alpha\}$ for every $0\leq t\leq k-1$, is either $1$ (in the case when no proper abelian power is guaranteed to start in any of $i+1$ consecutive positions) or the largest $k$ such that $(k-1)\|m\alpha\| \leq |I_1|-l_i=1-\|m\alpha\|-l_i$, i.e.,
\begin{equation*}
k=\left \lfloor \frac{1-l_{i}}{\|m\alpha\|}\right \rfloor. 
\end{equation*}
Consequently, Proposition \ref{pro:int} implies that
\begin{equation*}
  k_m^{(i)} = \max \left( 1, \left \lfloor \frac{1-l_{i}}{\|m\alpha\|}\right \rfloor \right). 
\end{equation*}
in this case.

Indeed, the case $\{m\alpha\} > 1/2$ is analogous. We can find $j$, with $0 \leq j \leq i$, such that
$\{\rho+(n-j)\alpha\} > 1-l_i$. Again such a point can be arbitrarily close to the point $1-l_i$. Thus, we need to find
the largest integer $k$ such that $1-l_i-(k-1)\|m\alpha\| \geq \|m\alpha\|$. The conclusion follows.
\end{proof}

\begin{example}
Take the Fibonacci word $f=s_{\phi-1,\phi-1}$, $m=10$ and $i=6$. We have $\|m\alpha\|\approx 0.180$ and $1-l_6\approx 0.764$, so from (\ref{eq:maxk2}) we get $k_m^{(i)}=4$. Indeed, using Corollary \ref{cor:gab2} we can compute the first values of the sequence $k_{10,n}$ (see Table \ref{tab:kmn}) and check that---at least in the first positions---there is a value $4$ in any interval of $i+1=7$ consecutive positions, but there are intervals of size $6$ in which no value $5$ is present, so that the exponent guaranteed for any $n$, with anticipation $6$, is equal to $4$.

The  values of  $k_{10}^{(i)}$ relative to $\alpha=\phi-1$ are reported in Table \ref{tab:kmi}.
\end{example}

\begin{table}
\centering  
\begin{small}
\begin{raggedright}
\begin{tabular}{c *{30}{@{\hspace{3.1mm}}c}}
$i$\hspace{2mm} &0 & 1 & 2 & 3 & 4 & 5 & 6 & 7 & 8 & 9 
\\
\hline \\
$k_{10}^{(i)}$\hspace{2mm} &1& 2& 3& 3& 4& 4& 4& 4& 4& 4
\\
\hline \rule[0pt]{0pt}{12pt}
\end{tabular}
\end{raggedright}\caption{\label{tab:kmi} The values of the guaranteed exponent $k_{10}^{(i)}$ for $0\leq i\leq 10$ in the Fibonacci word $f=s_{\phi-1,\phi-1}$.}
\end{small}
\end{table}

\section{Approximating Irrationals by Rationals and Abelian Repetitions}\label{sec:approx}

The results of the previous sections allow us to deal with abelian powers and abelian repetitions in a Sturmian word $s_{\alpha}$ by using classical results on the rational approximations of the irrational $\alpha$. 
Indeed, for any rational approximation $n/m$ of $\alpha$ such that $|n/m-\alpha|<1/km$, we have $|n-m\alpha|<1/k$, so $\|m\alpha\|<1/k$. Hence, by Theorem \ref{the:main1}, the Sturmian word $s_{\alpha}$ of angle $\alpha$ contains an abelian power of period $m$ and exponent $k$. Using this observation, we can translate classical results on the rational approximations of $\alpha$ into analogous properties on the abelian powers occurring in $s_{\alpha}$. 

In what follows, we recall some classical results from elementary number theory. For any notation not explicitly defined in this section we refer the reader to \cite{Hardy_and_Wright}.
Recall that every irrational number $\alpha$ can be uniquely written as a (simple) continued fraction as follows:
\begin{equation}\label{cf}
  \alpha = a_0 + \dfrac{1}{a_1 + \dfrac{1}{a_2 + \ldots}}
\end{equation}
where $a_{0}=\lfloor \alpha \rfloor$, and the infinite sequence $(a_{i})_{i\geq 0}$ is called the sequence of partial quotients of $\alpha$. The continued fraction expansion of $\alpha$ is usually denoted by its sequence of partial quotients as follows: $\alpha=[a_{0};a_{1},a_{2},\ldots ]$, and each its finite truncation $[a_{0};a_{1},a_{2},\ldots,a_{k}]$ is a rational number $n_{k}/m_{k}$ called the $k$th convergent to $\alpha$. We say that an irrational $\alpha=[a_{0};a_{1},a_{2},\ldots ]$ has bounded partial quotients if and only if the sequence $(a_{i})_{i\geq 0}$ is bounded.  

Since the Golden Ratio $\phi$ is defined by the equation $\phi=1+1/\phi$, we have from Equation \ref{cf} that $\phi=[1;\overline{1}]$ (we indicate a repeating period with a bar over the period) and therefore $\phi-1=[0;\overline{1}]$. The sequence $F_0=1, F_1=1, F_{j+1}=F_j+F_{j-1}$ for $j\geq 1$  is the well known sequence of Fibonacci numbers. The sequences of fractions $\left(F_{j+1}/F_j\right)_{j\geq 0}$ and $0,\left(F_{j}/F_{j+1}\right)_{j\geq 0}$ are the sequences of convergents to $\phi$ and $\phi-1$, respectively. 

The following classical result (see for example \cite[Theorem 6]{Lang}) states that the best rational  approximations of an irrational $\alpha$ are given by its convergents.

\begin{theorem}\label{theor:lang}
 For every irrational $\alpha$, if $n_{i}/m_{i}$ and $n_{i+1}/m_{i+1}$ are consecutive  convergents to $\alpha$, then $m_{i+1}$ is the smallest integer $m>m_{i}$ such that $\|m\alpha\|<\|m_{i}\alpha\|$.
\end{theorem}

From Theorem \ref{theor:lang} we directly have the following corollary.
 
\begin{corollary}\label{cor:cor}
Suppose that  $m_i>1$ and $m_{i+1}$ are consecutive denominators of convergents to $\alpha$. Then $\|m_i\alpha\|=\min_{1\leq m< m_{i+1}}\|m\alpha\|$.
\end{corollary}

Corollary \ref{cor:cor} has several consequences on the structure of abelian powers and abelian repetitions in Sturmian words.

The first one is given in the following proposition.

\begin{proposition}
 Let $s_{\alpha}$ be a Sturmian word of angle $\alpha$. For every integer $m>1$, let $k_{m}$ (resp. $k'_{m}$) be the maximum exponent of an abelian power (resp.~abelian repetition) of period $m$ in $s_{\alpha}$. Then the subsequence $(k_{m_{i}})_{i\geq 0}$ (resp.~$(k'_{m_{i}})_{i\geq 0}$), where the numbers $m_{i}$ are the denominators of the convergents to $\alpha$, is strictly increasing.
 \end{proposition}
 
\begin{proof}
The proposition is an immediate consequence of Theorem \ref{the:main1} and Corollary \ref{cor:cor}. 
\end{proof}

\begin{example}
In Table \ref{tab:fici3}, we give the first values of the sequence $\|m(\phi-1)\|$. Notice that the local minima correspond to the Fibonacci numbers (in bold), which are the denominators of the convergents to $\phi-1$.

From Corollary \ref{cor:cor} and Theorem \ref{the:main1}, we have that the local maxima of the sequence $k_{m}$ are precisely the values of $m$ that are Fibonacci numbers (see Table \ref{tab:fici2}).
 \end{example}
 
From Proposition \ref{pro:main} and Corollary \ref{cor:cor} we deduce the following.

\begin{proposition}\label{prop:unique}
 Let $m$ be a denominator of a convergent to $\alpha$. If $\{-m\alpha\}\geq \{-\alpha\}$ (resp.~if $\{-m\alpha\}< \{-\alpha\}$), then there is only one heavy (resp.~light) factor of length $m$, which starts and ends with the letter $\sa{a}$ (resp.~with the letter $\sa{b}$).
\end{proposition}

The previous proposition allows us to state that the abelian repetitions of period $m$, when $m$ is a denominator of a convergent to $\alpha$, have maximum head length and maximum tail length. More precisely, we have the following.

\begin{proposition}\label{prop:mec}
 Let $s_{\alpha} = a_0 a_1 a_2 \cdots$ be a Sturmian word of angle $\alpha$ and $m_{i}$ be a denominator of a
 convergent to $\alpha$. Let $w$ be an abelian power of period $m_{i}$ in $s_{\alpha}$ starting at position
 $n \geq m_{i} - 1$. Then this occurrence of $w$ can be extended to an abelian repetition of period $m_{i}$ with
 maximum head and tail length $m_{i}-1$.
\end{proposition}

\begin{proof}
  Let $w=a_{n}\cdots a_{n+m_{i}k-1}$ be an abelian power of period $m_{i}$ and exponent $k$ in $s_{\alpha}$. Suppose first that $w$ has maximum exponent $k_{m_{i}}$. We claim that the Parikh vectors of $a_{n-m_{i}+1}\cdots a_{n-1}$ and $a_{n+m_{i}k}\cdots a_{n+m_{i}k+m_{i}-2}$, both of length $m_{i}-1$, are contained in the Parikh vector $\PV$ of $a_{n}\cdots a_{n+m_{i}-1}$. This implies that the abelian repetition $a_{n-m_{i}+1}\cdots a_{n+m_{i}k+m_{i}-2}$ of period $m_{i}$ has maximum head length and maximum tail length, and the statement follows.

In order to prove the claim, let $\PV'$ be the Parikh vector of the factors of length $m_{i}$ of $s_{\alpha}$ that do not have Parikh vector $\PV$. Suppose that $\PV$ is the Parikh vector of the heavy (resp.~light) factors and that $\PV'$ is the Parikh vector of the light (resp.~heavy) factors. By the maximality of $k$, the Parikh vector of the two factors of length $m_{i}$, $a_{n-m_{i}}\cdots a_{n-1}$ and $a_{n+m_{i}k_{m_{i}}}\cdots a_{n+m_{i}k_{m_{i}}+m_{i}-1}$, respectively preceding and following the occurrence of $w$ in $s_{\alpha}$, must be $\PV'$ (we can extend $s_{\alpha}$ to the left by one letter if needed). 
By Proposition \ref{prop:unique}, there is a unique light (resp.~heavy) factor having Parikh vector $\PV'$.  
Moreover, this unique factor starts and ends with $\sa{b}$ (resp.~$\sa{a}$). Therefore, $a_{n-m_{i}}=a_{n+m_{i}k_{m_{i}}+m_{i}-1}$ is equal to $\sa{b}$ (resp.~to $\sa{a}$), and the claim is proved.

Finally, if the exponent of $w$ is not maximum, then the Parikh vector of the factor of length $m_{i}$ preceding (resp.~following) $w$ in $s_{\alpha}$ is either $\PV'$, if $w$ is a prefix (resp.~a suffix) of an abelian power of maximum exponent (and in this case the previous reasoning applies) or $\PV$. In both cases $w$ can be extended to an abelian repetition with maximum head length and maximum tail length. 
\end{proof}

\begin{corollary}
Let $s_{\alpha}$ be a Sturmian word of angle $\alpha$. Let $m_{i}$ be a  denominator of a convergent to $\alpha$. Then $k'_{m_{i}}=k_{m_{i}}+2-2/m_{i}$.
\end{corollary}

\begin{table}
\centering  
\begin{scriptsize}
\begin{raggedright}
\begin{tabular}{c *{30}{@{\hspace{1.4mm}}l}}
$m$\hspace{1mm} & {\bf 1} & {\bf 2} & {\bf 3} & 4 & {\bf 5} & 6 & 7 & {\bf 8} & 9 & 10 & 11 & 12 & {\bf 13} & 14 & 15 & 16 & 17 & 18 
\\
\hline \\
$\|m(\phi-1)\|$\hspace{1mm} &  {\bf 0.38} &  {\bf 0.24} & {\bf 0.15} & 0.47 & {\bf 0.09} & 0.29 & 0.33 & {\bf 0.06} & 0.44 & 0.18 & 0.20 & 0.42 & {\bf 0.03} & 0.35 & 0.27 & 0.11 & 0.49 & 0.13 
\\
\hline \rule[0pt]{0pt}{12pt}
\end{tabular}
\end{raggedright}\caption{\label{tab:fici3} The first values of the sequence $\|m(\phi-1)\|$. The values corresponding to the Fibonacci numbers are in bold.}
\end{scriptsize}
\end{table} 
 
In the context of ordinary powers, it is interesting to study the largest power occurring in a word. However, in the
abelian setting it does not make sense to study the analogous quantity since any Sturmian word contains abelian powers
of arbitrarily large exponent. Instead, we propose the following notion of abelian critical exponent, which measures
the maximum ratio between the exponent and the period of an abelian repetition.

\begin{definition}
  Let $s_{\alpha}$ be a Sturmian word of angle $\alpha$. For every integer $m>1$, let $k_{m}$ (resp. $k'_{m}$) be the
  maximum exponent of an abelian power (resp.~abelian repetition) of period $m$ in $s_{\alpha}$. The \emph{abelian
  critical exponent of $s_{\alpha}$} is defined as
  \begin{equation}\label{eq:1}
    \act(s_\alpha) = \limsup_{m \to \infty} \frac{k_{m}}{m} = \limsup_{m \to \infty} \frac{k'_{m}}{m}.
  \end{equation}
Notice that, indeed, the two superior limits  above  coincide for any Sturmian word $s_{\alpha}$ since
by definition one has $k_{m}\leq k'_{m}<k_{m}+2$ for every $m \geq 1$.
\end{definition}

Before studying abelian critical exponents further, we explore their connection to a number-theoretical concept known as
the Lagrange spectrum.

\begin{definition}
  Let $\alpha$ be a real number. The \emph{Lagrange constant of $\alpha$} is defined as
  \begin{equation*}
    \lambda(\alpha) = \limsup_{m\to\infty} (m\|m\alpha\|)^{-1}.
  \end{equation*}
\end{definition}

Let us briefly motivate the definition of the Lagrange constants. The famous Hurwitz's Theorem states that for every irrational $\alpha$ there exists infinitely many rational numbers $n/m$ such that $$\left|\alpha - \frac{n}{m}\right| < \frac{1}{\sqrt{5}m^2}$$ and, moreover, the constant $\sqrt{5}$ is best possible. Indeed, if $\alpha=\phi-1$, then for every $A>\sqrt{5}$ the inequality 
$$\left|\frac{n}{m}-\alpha\right|< \frac{1}{Am^2}$$
has only a finite number of solutions $n/m$.
 
%
%

For a general irrational $\alpha$, the infimum of the real numbers $\lambda$ such that for every $A>\lambda$ the inequality 
$\left|n/m-\alpha\right|< 1/Am^2$
has only a finite number of solutions $n/m$, is indeed the Lagrange constant $\lambda(\alpha)$ of $\alpha$. The set of all finite Lagrange
constants of irrationals is called the \emph{Lagrange spectrum} $L$. The Lagrange spectrum has been
extensively studied, see for instance \cite{Cusick_and_Flahive}. Yet its structure is still not completely understood. Markov \cite{Ma} proved that
$L \cap (-\infty, 3) = \{k_1 = \sqrt{5} < k_2 = \sqrt{8} < k_3 = \sqrt{221}/5 < \ldots\}$
where $k_n$ is a sequence of quadratic irrational numbers  converging to $3$ (so the beginning of $L$ is discrete). Then Hall \cite{Hall} proved  that $L$ contains a whole half line, and Freiman \cite{Freiman} determined the biggest half line that is contained
in $L$, which is $[c, +\infty)$, with
$$c=\frac{2221564096+283748\sqrt{462}}{491993569}=4.5278295661\ldots $$

Using the terminology of Lagrange constants, we have the following direct consequence of Theorem \ref{the:main1}.

\begin{theorem}
  Let $s_\alpha$ be a Sturmian word of angle $\alpha$. Then $\act(s_\alpha) = \lambda(\alpha)$. In other words, the
  abelian critical exponent of a Sturmian word is the Lagrange constant of its angle.
\end{theorem}

Let us next derive a (known) formula for the Lagrange constant of an irrational number. For this we need to recall some
elementary results on continued fractions. For full details see \cite[Chap.~X]{Hardy_and_Wright}.

Let $\alpha$ be a fixed irrational with continued fraction expansion $[a_0; a_1, a_2, \ldots]$. Let us set
$\alpha_i = [a_i; a_{i+1}, a_{i+2}, \ldots].$ Since $\alpha = [a_0; a_1, a_2, \ldots, a_i, \alpha_{i+1}]$, we have that
$$\alpha = \frac{\alpha_{i+1}n_i + n_{i-1}}{\alpha_{i+1}m_i + m_{i-1}}.$$ Therefore, by applying the identity
$m_i r_{i-1} - n_i m_{i-1} = (-1)^i$, we obtain that
\begin{equation}\label{eq:difference}
  \alpha - \frac{n_i}{m_i} = \frac{(-1)^i}{m_i(\alpha_{i+1}m_i + m_{i-1})},
\end{equation}
so $$(m_i\|m_i\alpha\|)^{-1} = \alpha_{i+1} + \frac{m_{i-1}}{m_i}.$$ By induction it is easy prove the
well-known fact that $m_{i-1}/m_i = [0;a_i,a_{i-1},\ldots,a_1]$. Consequently,
$$(m_i\|m_i\alpha\|)^{-1} = [a_{i+1};a_{i+2},\ldots]+[0;a_i,a_{i-1},\ldots,a_1].$$ Let $m$ be an integer
such that $m_i < m < m_{i+1}$ for some $i \geq 0$. By Theorem \ref{theor:lang}, we have
$\|m_i\alpha\| < \|m\alpha\|$, so $m_i\|m_i\alpha\| < m\|m\alpha\|$. Thus,
\begin{equation}\label{eq:lagrange_formula}
  \lambda(\alpha) = \limsup_{i\to\infty} (m_i\|m_i\alpha\|)^{-1} = \limsup_{i\to\infty} \left([a_{i+1};a_{i+2},\ldots]+[0;a_i,a_{i-1},\ldots,a_1]\right).
\end{equation}

\begin{definition}
  Let $\alpha$ and $\beta$ be two real numbers having continued fraction expansions $[a_0; a_1, a_2, \ldots]$ and
  $[b_0; b_1, b_2, \ldots]$ respectively. If there exists integers $N$ and $M$ such that $a_{N+i} = b_{M+i}$ for all
  $i \geq 0$, then we say that $\alpha$ and $\beta$ are \emph{equivalent}. In other words, two numbers are equivalent
  if their continued fraction expansions ultimately coincide.
\end{definition}

By Equation \eqref{eq:lagrange_formula} it is immediate that equivalent numbers have the same Lagrange constant. Equation
\eqref{eq:lagrange_formula} directly implies the following important proposition.

\begin{proposition}\label{prp:act_formula}
  Let $s_\alpha$ be a Sturmian word of angle $\alpha$. Then
  \begin{equation*}
    \act(s_\alpha) = \limsup_{i\to\infty} \left([a_{i+1};a_{i+2},\ldots]+[0;a_i,a_{i-1},\ldots,a_1]\right).
  \end{equation*}
\end{proposition}

We have thus obtained a formula for the abelian critical exponent of a Sturmian word in terms of the partial quotients
of its angle. A formula for the usual critical exponent of Sturmian words can also be expressed in terms of partial
quotients; see \cite{CaDel00,Justin2001363,Damanik200223,Pelto}.

Proposition \ref{prp:act_formula} enables us to study the abelian critical exponent of Sturmian words. The first
application is the following result. Recall that an infinite word is $\beta$-power-free if it does not contain
repetitions of exponent $\beta$ or larger.

\begin{theorem}\label{theor:act_finite}
  Let $s_\alpha$ be a Sturmian word of angle $\alpha$. The following are equivalent:
  \begin{enumerate}[(i)]
    \item $\act(s_\alpha)$ is finite,
    \item $s_\alpha$ is $\beta$-power-free for some $\beta \geq 2$,
    \item $\alpha$ has bounded partial quotients.
  \end{enumerate}
\end{theorem}
\begin{proof}
  It is evident from Proposition \ref{prp:act_formula} that $\act(s_\alpha)$ is finite if and only if $\alpha$ has
  bounded partial quotients. By a well-known result, $s_\alpha$ is $\beta$-power-free for some $\beta \geq 2$ if and
  only if $\alpha$ has bounded partial quotients; see \cite{Mi89}.
\end{proof}

\begin{theorem}\label{theor:sqrt5}
  For every Sturmian word $s_{\alpha}$ of angle $\alpha$, we have $\act(s_\alpha) \geq \sqrt{5}$. Moreover,
  $\act(s_\alpha) = \sqrt{5}$ if and only if $\alpha$ is equivalent to $\phi-1$. In particular, the abelian critical
  exponent of the Fibonacci word is $\sqrt{5}$.
\end{theorem}
\begin{proof}
  It is clear from Proposition \ref{prp:act_formula} that $\act(s_\alpha)$ is as small as possible when $\alpha$ is
  equivalent to $\phi = [1;\overline{1}]$. It is straightforward to compute that $\lambda(\phi) = \sqrt{5}$, so
  $\act(s_\alpha) \geq \act(s_{1-\phi,1-\phi}) = \sqrt{5}$ for all Sturmian words $s_\alpha$.

  What is left is to prove is that if $\act(s_\alpha) = \sqrt{5}$, then $\alpha$ is equivalent to $\phi-1$. Suppose
  that $\alpha = [0;a_1,a_2,\ldots]$ and $\act(s_\alpha) = \sqrt{5}$. If $a_i \geq 3$ for infinitely many $i$, then
  clearly $\act(s_\alpha) \geq 3 > \sqrt{5}$. Thus, there exists $M > 0$ such that $a_i < 3$ for all $i \geq M$. We are
  left with two cases: either $a_i = 1$ for only finitely many $i$ or the sequence $(a_i)$ takes values $1$ and $2$
  infinitely often; otherwise we are done.
  
  Suppose first that $a_i = 1$ for finitely many $i$. It follows that $(a_i)$ eventually takes only the value $2$, so
  $\alpha$ is equivalent to $\sqrt{2} = [2;\overline{2}]$. Therefore, $\lambda(\alpha) = \lambda(\sqrt{2})$. It is
  routine computation to show that $\lambda(\sqrt{2}) = \sqrt{8}$, so $\act(s_\alpha) = \sqrt{8} > \sqrt{5}$; a
  contradiction.

  Assume finally that the sequence $(a_i)$ takes values $1$ and $2$ infinitely often. It follows that the sequence
  $(a_i)$ contains either of the patterns $2,1,1$ or $2,1,2$ infinitely often. Since an odd convergent of a number
  $\beta$ is always strictly less than $\beta$, it follows that $$[2,1,1,a_k,a_{k+1},\ldots] > [2,1,1] = \frac52$$ and
  $$[2,1,2,a_k,a_{k+1},\ldots] > [2,1,2] = \frac83.$$ Thus, $\act(s_\alpha) \geq 5/2 > \sqrt{5}$, which is impossible.
\end{proof}

In general, two numbers having the same Lagrange constant need not be equivalent. Any two numbers having unbounded partial
quotients have Lagrange constant $\infty$, but obviously such numbers are not necessarily equivalent.

As a consequence of Theorem \ref{theor:sqrt5} we have the following.

\begin{corollary}\label{cor:fil_new}
Let $s_{\alpha}$ be a Sturmian word of angle $\alpha$. For every $\delta>0$ there exists an increasing sequence of integers $(m_i)_{i\geq 0}$ such that for every $i$ there is in $s_\alpha$ an abelian power of period $m_i$ and length greater than $(\sqrt{5}-\delta)m_i^2$ (i.e., with exponent greater than $(\sqrt{5}-\delta)m_i$).
\end{corollary}

Notice that the result of the previous corollary about abelian powers is in sharp contrast with the analogous situation
for ordinary repetitions. Indeed, it is known (see \cite{MignosiPirillo}) that there exist Sturmian words that are
$\beta$-power-free for any real number $\beta \geq 2+\phi$, so with respect to both the length and the exponent the
difference with the abelian setting is of one order of magnitude.

\begin{remark}
  By Proposition \ref{prp:act_formula} it is in principle possible to explicitly compute the abelian critical exponent
  for a given Sturmian word. This is especially true if the angle $\alpha$ is a quadratic irrational, as then the
  continued fraction expansion of $\alpha$ is ultimately periodic
  \cite[Chap.~X, Theorem~176, Theorem~177]{Hardy_and_Wright}. Observe that this implies that the partial quotients of
  the expansion are bounded, so by Theorem \ref{theor:act_finite} the abelian critical exponent is finite in this case.
  Notice that if a Sturmian word is a fixed point of a substitution, then its angle is a quadratic irrational called a
  Sturm number \cite{yasutomi}.

  Suppose for an example that $\alpha = [0;\overline{2,1}]$. It is routine to compute that
  $[0;\overline{2,1}] = (-1+\sqrt{3})/2$ and $[0;\overline{1,2}] = -1+\sqrt{3}$. Combining the fact that odd
  convergents of $[0;\overline{1,2}]$ approximate $[0;\overline{1,2}]$ from below with the fact that
  $[2;\overline{1,2}] > [1;\overline{2,1}] + 1$, it follows from Proposition \ref{prp:act_formula} that
  $A(s_\alpha) = [2;\overline{1,2}] + [0;\overline{1,2}] = 2\sqrt{3} \approx 3.46$ for a Sturmian word $s_\alpha$ of
  angle $\alpha$.
\end{remark}

\section{Abelian Repetitions in the Fibonacci Word}\label{sec:abrepFibo}

In this section we apply our results to the Fibonacci word and study in detail its abelian powers and
repetitions. We begin with the following simple observation.

\begin{proposition}\label{pro:max}
  The maximum exponent of an abelian power of period $F_{j}$, $j>0$, in the Fibonacci word is equal to
  $\lfloor \phi F_j+F_{j-1} \rfloor$.
\end{proposition}
\begin{proof}
  The sequence of the denominators of the convergents of $\phi-1$ coincide with the sequence of Fibonacci numbers.
  Therefore, it is an immediate consequence of Equation \eqref{eq:difference} that
  $$\|F_{j}(\phi-1)\|= \frac{1}{\phi F_{j}+F_{j-1}}.$$ The claim follows now from Theorem \ref{the:main1}.
\end{proof}

Next we turn our attention to the prefixes of the Fibonacci word which are abelian repetitions. Consider such a prefix
of abelian period $m$. The head of the abelian decomposition of this prefix has length at most $m-1$. Therefore, in
order to find the longest abelian repetition of period $m$ that is a prefix, we have to check the maximum length of a
compatible head of all the abelian powers that start at positions $i$ such that $0 \leq i \leq m-1$.

\begin{proposition}\label{pro:maximal_exponent}
  Let $j>1$. In the Fibonacci word $f$, the longest abelian power of period $F_j$ starting at a position $i< F_j$
  has an occurrence starting at position $F_j-1$ and has exponent $$\lfloor \phi F_j+F_{j-1} \rfloor -1=
  \begin{cases}
   F_{j+1}+F_{j-1} -1 &  \mbox{ if $j$ is even;}\\
   F_{j+1}+F_{j-1} -2 &  \mbox{ if $j$ is odd.}
  \end{cases}
  $$
\end{proposition}
\begin{proof}
For simplicity we consider the word $s_{\phi-1,0}$ and show that in $s_{\phi-1,0}$ the longest abelian power having period $F_j$ starting at a position $i$ such that $i\leq F_j$ has an occurrence starting at position $F_j$ and has the claimed exponent. By Theorem \ref{the:maingab}, an abelian power with period $F_{j}$ starting at position $F_{j}$ in $s_{\phi-1,0}$ has exponent $k$ if and only if $$\{F_{j}(\phi-1)\}<1-k\{F_{j}(\phi-1)\} \ \mbox{ or } \ \{F_{j}(\phi-1)\}>k\{-F_{j}(\phi-1)\}.$$ 

Suppose that the first case holds (the other case is analogous); then $j$ is even. We hence have
\begin{equation}\label{eq:uno}
 \{F_{j}(\phi-1)\}<\frac{1}{k+1}.
\end{equation}
As in the proof of Proposition \ref{pro:max}, we derive from Equation \eqref{eq:difference} that
$$\{F_{j}(\phi-1)\}=\frac{1}{\phi F_{j}+F_{j-1}}.$$
Therefore, the largest integer $k$ for which \eqref{eq:uno} holds is $\lfloor \phi F_j+F_{j-1} \rfloor -1$.

Then, the abelian power of period $F_{j}$ starting at position $F_{j}-1$ in the Fibonacci word has exponent $\lfloor \phi F_j+F_{j-1} \rfloor -1$. By Corollary \ref{cor:cor} any abelian power starting at a position $i$ such that $i<F_{j}-1$ has a smaller exponent, so the proof is complete if we derive the formula of the statement. By Equation \eqref{eq:difference} 
\begin{equation*}
  \phi F_j - F_{j+1} = \frac{(-1)^j}{\phi F_j + F_{j-1}},
\end{equation*}
so we obtain that
\begin{equation*}
  \lfloor \phi F_j \rfloor = F_{j+1} + 
  \begin{cases}
    0 & \text{ if $j$ is even;}\\
    -1 & \text{ if $j$ is odd.}
  \end{cases}
\end{equation*}
This gives the formula of the statement.
\end{proof}

The following theorem provides a formula for computing the length of the longest abelian repetition occurring as a prefix in the Fibonacci word.

\begin{theorem}\label{pro:longest_prefix}
Let $j>1$. The longest prefix of the Fibonacci word that is an abelian repetition of period $F_j$ has length 
$$\lp(F_{j}) =
\begin{cases}
 F_j( F_{j+1}+F_{j-1} +1)-2 &  \mbox{ if $j$ is even;}\\
 F_j( F_{j+1}+F_{j-1} )-2   &  \mbox{ if $j$ is odd.}
\end{cases}
$$
\end{theorem}
\begin{proof}
Let $w$ be the abelian power in $f$ of period $F_j$ having maximum exponent described in Proposition
\ref{pro:maximal_exponent}. By Proposition \ref{prop:mec} this occurrence of $w$ can be extended to an abelian
repetition with maximum head and tail length. The claim thus follows from the formula of Proposition
\ref{pro:maximal_exponent}.
\end{proof}

\begin{proposition}\label{corr:last}
Let $j>1$ and $k_{j}$ be the maximum exponent of a prefix of the Fibonacci word that is an abelian repetition of period $F_{j}$. Then $$\lim_{j\to \infty}\frac{k_{j}}{F_{j}}=\sqrt{5}.$$
\end{proposition}

\begin{proof} 
We proved in Theorem \ref{theor:sqrt5} that over a bigger set of repetitions the superior limit tends to $\sqrt{5}$.
Therefore, the limit must be smaller than or equal to $\sqrt{5}$. The equality follows from the fact that the
sequence $(F_{j+1}+F_{j-1})/F_j$ converges to $\sqrt{5}$. Indeed, since the sequence $F_{j+1}/F_j$ converges to $\phi$
and the sequence $F_{j-1}/F_j$ converges to $\phi-1$, the sequence $$\frac{F_{j+1}+F_{j-1}}{F_j}$$ converges to
$\phi + \phi-1= \sqrt{5}.$ 
\end{proof}

\begin{figure}[ht!]
\centering
\fbox{
\includegraphics[scale=0.8]{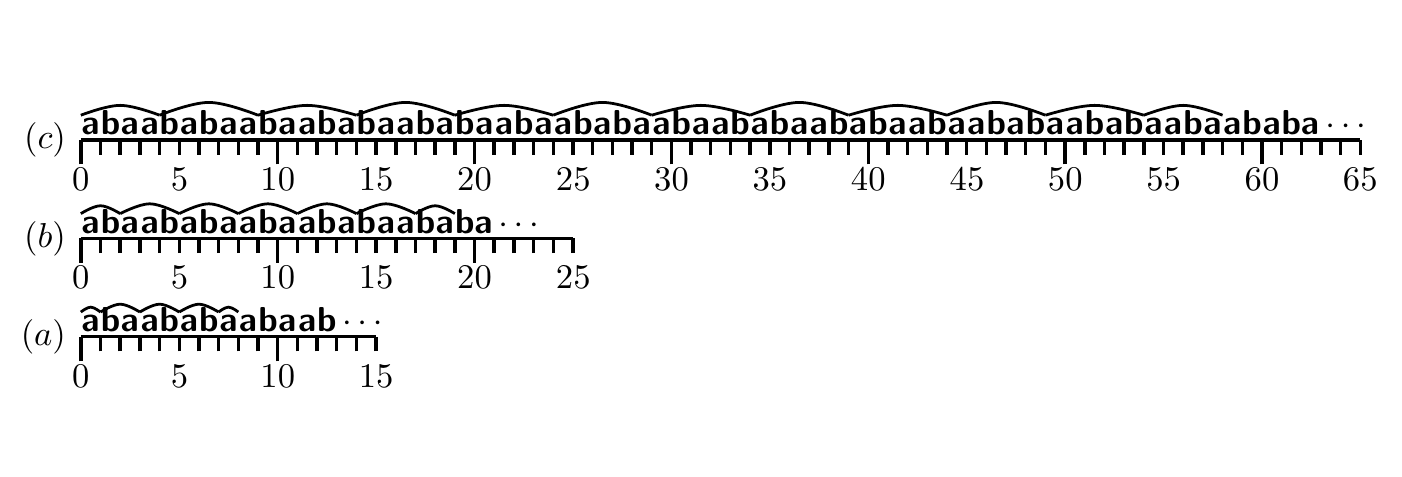} 
}\caption{Longest abelian repetition of period $m$ that is a  prefix of the  Fibonacci word for $m=2,3,5$. 
$(a)$ For $m=2$, the longest abelian repetition has length $8=1+3m+1$. 
$(b)$ For $m=3$, the longest abelian repetition has length $19=2+5m+2$.
$(c)$ For $m=5$, the longest abelian repetition has length $58=4+10m+4$. }
\label{Fig:prefixlenghts}
\end{figure}

In \figurename~\ref{Fig:prefixlenghts} we give a graphical representation of the longest prefix of the Fibonacci word that is an abelian repetition of period $m$ for $m=2,3,5$. In Table \ref{tab:fici} we give the length $\lp(F_{j})$ of the longest prefix of the Fibonacci word that is an abelian repetition of period $F_{j}$, for $j=2,\ldots, 11$, computed using the formula of Theorem \ref{pro:longest_prefix}. We also give the distance between $\sqrt{5}$ and the ratio between the maximum exponent $\lp(F_{j})/F_{j}$ of a prefix of the Fibonacci word having abelian period $F_{j}$ and the number $F_{j}$.

\begin{table}
\centering  
\begin{small}
\begin{raggedright}
\begin{tabular}{c *{30}{@{\hspace{3.1mm}}l}}
$j$\hspace{2mm} & 2& 3& 4& 5& 6& 7& 8& 9& 10& 11 
\\
\hline \\
$F_{j}$\hspace{2mm} & 2& 3& 5& 8& 13& 21& 34& 55& 89& 144 
\\
\hline \\
$\lp(F_{j})$ \hspace{2mm}   & 8 & 19  & 58  & 142  & 388  & 985  & 2616\hspace{1ex} & 6763  & 17798 & 46366 
\\
 \hline\\
$|\sqrt{5}-k_{j}/F_{j}|\times 10^2$ \hspace{2mm}   & $23.6 $ & $12.5 $ & $8.393 $ & $1.732 $ & $5.98 $ & $0.25 $ & $2.69 $ & $0.037 $ & $1.087$ & $0.005 $ 
\\
\hline \rule[0pt]{0pt}{12pt}
\end{tabular}
\end{raggedright}\caption{\label{tab:fici}The length of the longest prefix ($\lp(F_{j})$) of the Fibonacci word having abelian period $F_{j}$ for $j=2,\ldots, 11$.  The table also reports rounded distances (multiplied by $10^2$) between $\sqrt{5}$ and the ratio between the exponent $\lp(F_{j})/F_{j}$ of the longest prefix of the Fibonacci word having abelian period $F_{j}$ and the number $F_{j}$ (see Proposition~\ref{corr:last}).}
\end{small}
\end{table}

Next we extend a classical result on the periods of the factors of the Fibonacci word to the abelian setting. Currie
and Saari proved the following \cite{CuSa09}.

\begin{proposition}
 The minimum period of any factor of the Fibonacci infinite word is a Fibonacci number.
\end{proposition}

We prove an analogous result for abelian periods: the minimum abelian period of every factor of the Fibonacci word is
a Fibonacci number (Theorem \ref{the:abperFib}). For the proof we need two lemmas.

\begin{lemma}\label{lem:ratio}
  For all $j > 2$ we have
  \begin{equation*}
    \frac{\|F_{j-1}(\phi-1)\|}{\|F_j(\phi-1)\|} = \phi \ \text{ and } \ \frac{\|F_{j-2}(\phi-1)\|}{\|F_j(\phi-1)\|} = 1 + \phi.
  \end{equation*}
\end{lemma}
\begin{proof}
  It is straightforward to verify that $\|F_1(\phi-1)\|/\|F_2(\phi-1)\| = \phi$. By applying induction and Equation
  \eqref{eq:difference} we obtain that
  \begin{align*}
    \frac{\|F_{j-1}(\phi-1)\|}{\|F_j(\phi-1)\|} &= \frac{\phi F_j + F_{j-1}}{\phi F_{j-1} + F_{j-2}}
    = \frac{\phi F_{j-1} + \phi F_{j-2} + F_{j-2} + F_{j-3}}{\phi F_{j-1} + F_{j-2}} \\
    &= 1 + \frac{\|F_{j-1}(\phi-1)\|}{\|F_{j-2}(\phi-1)\|} = 1 + \frac{1}{\phi} = \phi.
  \end{align*}
  Similarly,
  \begin{equation*}
    \frac{\|F_{j-2}(\phi-1)\|}{\|F_j(\phi-1)\|} = 1 + \frac{\|F_{j-2}(\phi-1)\|}{\|F_{j-1}(\phi-1)\|} = 1 + \phi.
  \end{equation*}
\end{proof}

For the next lemma we need the following corollary of the famous Three Distance Theorem; see, e.g., \cite{AlBe}.

\begin{proposition}\label{prp:tdt}
  Let $\alpha = [0;a_1,a_2,\ldots]$ be an irrational number and $(n_k/m_k)$ be its sequence of convergents. If $k > 1$, then the
  lengths of the $m_k$ subintervals $L_0(m_k-1)$, $L_1(m_k-1)$, $\ldots$, $L_{m_k-1}(m_k-1)$ of the torus take two
  values: $\|m_{k-1}\alpha\|$ or $\|((a_k - 1)m_{k-1} + m_{k-2})\alpha\|$.
\end{proposition}

\begin{lemma}\label{lem:kmi_fibonacci}
  For the Fibonacci word we have $k_{F_j}^{(F_j - 1)} = F_{j+1} + F_{j-1} - 3$ for all $j > 1$.
\end{lemma}
\begin{proof}
  By Proposition \ref{prp:tdt} and Corollary \ref{cor:cor} the maximum size of the intervals of factors of length
  $F_j - 1$ is $\|F_{j-2}(\phi-1)\|$. Therefore, Theorem \ref{theor:kmi}, Lemma \ref{lem:ratio} and Equation
  \eqref{eq:difference} imply that
  \begin{equation*}
    k_{F_j}^{(F_j-1)} = \max\left( 1, \left\lfloor \phi F_j + F_{j-1} - 1 - \phi \right\rfloor \right).
  \end{equation*}
  Further, Equation \eqref{eq:difference} implies that
  \begin{equation*}
    \phi F_j = F_{j+1} + \frac{(-1)^j}{\phi F_j + F_{j-1}}.
  \end{equation*}
  Since $1/(\phi F_j + F_{j-1})$ is at most $2-\phi$, it follows that
  \begin{equation*}
    F_{j+1} - 2 \leq \phi F_j - \phi \leq F_{j+1} + 2 - 2\phi
  \end{equation*}
  Since $3 < 2\phi < 4$, we have that $\left\lfloor \phi F_j - \phi \right\rfloor = F_{j+1} - 2$. The claim follows.
\end{proof}

\begin{theorem}\label{the:abperFib}
  The minimum abelian period of any factor of the Fibonacci word is a Fibonacci number.
\end{theorem}

\begin{proof}
Let $w$ be a factor of the Fibonacci infinite word $f$, and suppose that $w$ has an abelian period $m>0$. We will show
that $w$ has also period $F_{n}$ where $F_{n}$ is the largest Fibonacci number such that $F_n \leq m$. If $m = F_n$,
then there is nothing to prove, so we can suppose that $F_n < m < F_{n+1}$. In particular, we have that $n \geq 3$. We
will show that given a suitable occurrence of $w$ in $f$ there is an earlier occurrence of an abelian repetition $w'$
of period $F_n$ such that $w$ is a factor of $w'$. The conclusion follows then from Lemma \ref{lem:ap}.

Suppose that $w$ occurs in $f$ at position $i$. By Theorem \ref{theor:kmi} there is an abelian power of period $F_n$ of
length $F_n k_{F_n}^{(F_n - 1)}$ starting at position $i + j$ for some $j$ such that $0 \leq j \leq F_n - 1$. By
Proposition \ref{prop:mec} this abelian power can be extended to an abelian repetition with maximum head and tail
length $F_n - 1$, so we only need to ensure that this abelian repetition is long enough to have $w$ as a factor. Since
$w$ has length at most $m(k_m + 2) - 2$, we thus need to establish that
\begin{equation*}
  m(k_m + 2) - 2 \leq F_n(k_{F_n}^{(F_n - 1)} + 1) - 1.
\end{equation*}
By Lemma \ref{lem:kmi_fibonacci} this inequality holds if and only if the inequality
\begin{equation}\label{eq:equiv1}
  m(k_{m}+2) \leq F_n(F_{n+1} + F_{n-1} - 2) + 1
\end{equation}
is satisfied. The rest of the proof consists of showing that \eqref{eq:equiv1} holds.

First we derive the following upper bound on $m(k_m + 2)$:
\begin{equation}\label{eq:tech1}
  m(k_{m}+2) < F_{n+1}(F_{n-1}+F_{n-3}+2).
\end{equation}
Let us first show that $\|m(\phi-1)\|>\|F_{n-2}(\phi-1)\|$. By Theorem \ref{theor:lang} we have
$$\|F_{n-2}(\phi-1)\|=\min_{i< F_{n-1}}\|i(\phi-1)\|.$$ Equation \eqref{eq:difference} implies that either
$\{F_n(\phi-1)\}<1/2$ and $\{F_{n+1}(\phi-1)\}>1/2$ or $\{F_n(\phi-1)\}>1/2$ and $\{F_{n+1}(\phi-1)\}<1/2$. Suppose
first that $\{m(\phi-1)\}<1/2$. If $\{F_n(\phi-1)\}<1/2$, then we have
 \begin{eqnarray*}
 \|m(\phi-1)\|&=&\|m(\phi-1)\|-\|F_{n}(\phi-1)\|+\|F_{n}(\phi-1)\| \\ 
&=&  \|(m - F_{n})(\phi-1)\|+\|F_{n}(\phi-1)\| \\
&=&  \|(F_{n}-m)(\phi-1)\|+\|F_{n}(\phi-1)\| \\
&>&  \|F_{n-2}(\phi-1)\|+\|F_{n}(\phi-1)\| \\
&>&  \|F_{n-2}(\phi-1)\|.
 \end{eqnarray*}
If instead $\{F_n(\phi-1)\}>1/2$, then we can apply the same manipulation with $F_{n+1}$ in place of $F_n$. Indeed,
since the difference of $F_{n+1}$ and $F_n$ is $F_{n-1}$ and since $F_n<m<F_{n+1}$, we have that $F_{n+1} -m <F_{n-1}$,
so we can still apply Theorem \ref{theor:lang} to derive that $\|(F_{n+1}-m)(\phi-1)\| \geq \|F_{n-2}(\phi-1)\|$. The
case $\{m(\phi-1)\}>1/2$ is symmetric. Thus, we have shown that $\|m(\phi-1)\|>\|F_{n-2}(\phi-1)\|$.
Therefore, by Equation \eqref{eq:difference} we have $k_m < \phi F_{n-2} + F_{n-3}$. Again, by applying Equation
\eqref{eq:difference} as in the proof of Lemma \ref{lem:kmi_fibonacci}, we obtain that
$k_m \leq F_{n-1} + F_{n-3}$. As $m < F_{n+1}$, inequality \eqref{eq:tech1} follows.

By the inequality \eqref{eq:tech1}, in order to establish the inequality \eqref{eq:equiv1}, it is sufficient to
show that
\begin{equation*}
  F_{n+1}(F_{n-1}+F_{n-3}+2) \leq F_n(F_{n+1} + F_{n-1} - 2).
\end{equation*}
This inequality, however, is easily seen to be true whenever $F_{n-1}+F_{n-3}+2 \leq F_n$, that is, when $n \geq 6$.
By a direct computation it can be seen that the above inequality holds also for $n = 5$. Suppose then that $n = 4$. We
proved above that $k_m \leq F_{n-1} + F_{n-3} = 4$. Plugging the estimates $k_m \leq 4$ and $m \leq 7$ into
\eqref{eq:equiv1} shows that the conclusion holds also in this case. Suppose finally that $n = 3$, that is, $m = 4$.
Now $k_m = 2$, and a direct substitution to the inequality \eqref{eq:equiv1} shows that the conclusion holds. This ends
the proof.
\end{proof}

\begin{corollary}\label{cor:prefix}
The minimum abelian period of any prefix of the Fibonacci word is a Fibonacci number.
\end{corollary}

\begin{remark}
  Theorem \ref{the:abperFib} does not generalize to hold for every Sturmian word. Consider for instance Sturmian words
  of angle $\alpha = [0;\overline{2,1}] = (\sqrt{3}-1)/2$. It can be verified that the factor
  $$\sa{aabab} \cdot \sa{aabaabaababaabaabaababaabaabaa} \cdot \sa{babaa}$$ starting at position $35$ of
  $s_{\alpha,\alpha}$ is an abelian repetition of minimum period $6$ with maximum head and tail length. However, the
  number $6$ is not a denominator of a convergent of $\alpha$ since the sequence of convergents starts
  $0,1/2,1/3,3/8,\ldots$
\end{remark}

Recall that the (finite) Fibonacci words are defined by $f_{0}=\sa{b}$, $f_{1}=\sa{a}$ and $f_{j+1}=f_{j}f_{j-1}$ for every $j > 1$. Hence, for every $j\geq 0$, we have $|f_{j}|=F_{j}$.

From Corollary \ref{cor:prefix}, we know that every finite Fibonacci word has an abelian period that is a Fibonacci number. The following theorem, stated without proof in \cite{FiLaLeLeMiPG13}, provides an explicit formula for the minimum abelian period of the finite Fibonacci words.

\begin{theorem}
For $j \geq 3$, the minimum abelian period of the word $f_j$ is the $n$th Fibonacci number $F_n$, where  
$$n =
\begin{cases}
 \lfloor{j/2}\rfloor &  \mbox{ if $j = 0, 1, 2\mod{4}$;}\\
 \lfloor{j/2}\rfloor +1&  \mbox{ if  $ j  = 3\mod{4}$.}
\end{cases}
 $$
\end{theorem}

\begin{proof}
From Corollary \ref{cor:prefix} using Theorem \ref{pro:longest_prefix}, it is sufficient to find the smallest integer
$n$ such that $\lp(F_{n})$ is greater than or equal to $F_{j}$. In other words, we need to find the smallest integer
$n$ such that
\begin{equation*}
  F_n \left(F_{n+1} + F_{n-1} + \gamma \right) - 2 \geq F_j
\end{equation*}
where $\gamma$ equals $1$ if $n$ is even and $0$ if $n$ is odd.

We need the following well-known formula:
\begin{equation}\label{eq:fibo_2j}
  F_j(F_{j+1} + F_{j-1}) = F_{2j+1}.
\end{equation}
This identity follows easily from the matrix identity
\begin{equation*}
  \begin{pmatrix}
    1 &1 \\
    1 &0
  \end{pmatrix}^j
  = \begin{pmatrix}
      F_j     & F_{j-1} \\
      F_{j-1} & F_{j-2}
    \end{pmatrix}
\end{equation*}
and the fact that $A^n A^m = A^{n+m}$ for a matrix $A$.

It is now straightforward to verify the claim using Equation \eqref{eq:fibo_2j}. We will prove the claim in the case
that $j = 2 \mod 4$; the other cases are similar. Choose $n = \lfloor j/2 \rfloor$. Now $n$ is odd and $2n+1 = j+1$, so
by the Equation \eqref{eq:fibo_2j} we need to verify that $F_{j+1} - 2 \geq F_j$, which is clearly true. Choose then
$n = \lfloor j/2 \rfloor - 1$. Then $n$ is even and $2n + 1 = j - 1$, so $F_{2n+1} + F_n - 2 \geq F_j$ if and only if
$F_{\lfloor j/2 \rfloor - 1} - 2 \geq F_{j-2}$. This latter inequality, however, cannot hold as $F_{4k} \geq F_{2k}$
for all $k \geq 0$. This shows that the value $\lfloor j/2 \rfloor$ is minimal in this case.
\end{proof}

\begin{example}
The minimum abelian period of the word $f_{4}=\sa{abaab}$ is $2=F_{2}=F_{\lfloor{4/2}\rfloor}$, since $f_{4}=\sa{a}\cdot \sa{ba} \cdot \sa{ab}$; the minimum abelian period of $f_{5}=\sa{a}\cdot\sa{ba}\cdot\sa{ab}\cdot\sa{ab}\cdot \sa{a}$ is $2=F_{2}=F_{\lfloor{5/2}\rfloor}$; the minimum abelian period of $f_{6}=\sa{aba}\cdot\sa{aba}\cdot\sa{baa}\cdot\sa{baa}\cdot\sa{b}$ is $3=F_{3}=F_{\lfloor{6/2}\rfloor}$; the minimum abelian period of $f_{7}=\sa{abaab}\cdot\sa{abaab}\cdot\sa{aabab}\cdot\sa{aabab}\cdot\sa{a}$ is $5=F_{4}=F_{1+\lfloor{7/2}\rfloor}$.
In Table~\ref{tab:val} we report the minimum abelian periods of the first Fibonacci words. 
\end{example}

\begin{table}
\centering  
\begin{small}
\begin{raggedright}
\begin{tabular}{c *{30}{@{\hspace{2.1mm}}l}}
 $j$\hspace{2mm}  & 3\hspace{1ex} & 4\hspace{1ex} & 5\hspace{1ex} & 6\hspace{1ex} & 7\hspace{1ex} & 8\hspace{1ex} & 9\hspace{1ex} & 10 & 11 & 12 & 13 & 14 & 15 & 16 \\
 \hline \rule[-6pt]{0pt}{22pt}
abelian period~of $f_{j}$\hspace{2mm}  & $F_{2}$ & $F_{2}$ & $F_{2}$ & $F_{3}$ & $F_{4}$ & $F_{4}$ & $F_{4}$ & $F_{5}$ & $F_{6}$ & $F_{6}$ & $F_{6}$ & $F_{7}$ & $F_{8}$ & $F_{8}$\\
\hline \rule[-2pt]{0pt}{12pt}
\end{tabular}
\end{raggedright}\caption{\label{tab:val}The minimum abelian periods of the Fibonacci words $f_{j}$ for $j=3,\ldots, 16$.}
\end{small}
\end{table}

\subsection*{Acknowledgements}

This work has been partially supported by Italian MIUR Project PRIN~2010LYA9RH, ``Automi e Linguaggi Formali: Aspetti Matematici e Applicativi''.


\begin{thebibliography}{10}

\bibitem{AB1}
B.~Adamczewski and Y.~Bugeaud.
\newblock On the complexity of algebraic numbers I. Expansions in integer bases.
\newblock {\em Annals of Math.}, 165:547--565, 2007.

\bibitem{AB2}
B.~Adamczewski and Y.~Bugeaud.
\newblock On the complexity of algebraic numbers  II. Continued fractions.
\newblock {\em Acta Math.}, 195:1--20, 2005.

\bibitem{AlBe}
P.~Alessandri and V.~Berth\'e.
\newblock Three distance theorems and combinatorics on words.
\newblock {\em Enseig. Math.}, 44:103--132, 1998.

\bibitem{AKP2012}
S.~Avgustinovich, J.~Karhum{\"a}ki, and S.~Puzynina.
\newblock On abelian versions of {C}ritical {F}actorization {T}heorem.
\newblock {\em RAIRO Theor. Inform. Appl.}, 46:3--15, 2012.

\bibitem{Be99}
J.~Berstel.
\newblock On the index of {S}turmian words.
\newblock In J.~Karhum{\"a}ki, H.~Maurer, G.~P\v{a}un, and G.~Rozenberg,
  editors, {\em Jewels are Forever}, pages 287--294. Springer Berlin
  Heidelberg, 1999.

\bibitem{BerstelRecent07}
J.~Berstel.
\newblock {S}turmian and episturmian words (a survey of some recent results).
\newblock In S.~Bozapalidis and G.~Rohonis, editors, {\em CAI 2007}, volume
  4728 of {\em Lecture Notes in Comput. Sci.}, pages 23--47. Springer, 2007.

\bibitem{Berstel-Reutenauersurvey}
J.~Berstel, A.~Lauve, C.~Reutenauer, and F.~Saliola.
\newblock {\em Combinatorics on Words: Christoffel Words and Repetition in
  Words}, volume~27 of {\em CRM monograph series}.
\newblock American Mathematical Society, 2008.

\bibitem{BeHoZa06}
V.~Berth\'e, C.~Holton, and L.~Q. Zamboni.
\newblock Initial powers of {S}turmian sequences.
\newblock {\em Acta Arith.}, pages 315--347, 2006.

\bibitem{Bucci201225}
M.~Bucci, A.~De~Luca, and L.~Q. Zamboni.
\newblock Some characterizations of {S}turmian words in terms of the
  lexicographic order.
\newblock {\em Fund. Inform.}, 116(1-4):25--33, 2012.

\bibitem{CaDel00}
A.~Carpi and A.~de~Luca.
\newblock Special factors, periodicity, and an application to {S}turmian words.
\newblock {\em Acta Inform.}, 36(12):983--1006, 2000.

\bibitem{CRSZ2010}
J.~Cassaigne, G.~Richomme, K.~Saari, and L.~Zamboni.
\newblock Avoiding {A}belian powers in binary words with bounded {A}belian
  complexity.
\newblock {\em Int. J. Found. Comput. Sci.}, 22(4):905--920, 2011.

\bibitem{CI2006}
S.~Constantinescu and L.~Ilie.
\newblock {F}ine and {W}ilf's theorem for abelian periods.
\newblock {\em Bull. Eur. Assoc. Theoret. Comput. Sci. EATCS}, 89:167--170,
  2006.

\bibitem{Cummings_weakrepetitions}
L.~J. Cummings and W.~F. Smyth.
\newblock Weak repetitions in strings.
\newblock {\em {J. Combin. Math. Combin. Comput.}}, 24:33--48, 1997.

\bibitem{CuSa09}
J.~D. Currie and K.~Saari.
\newblock Least periods of factors of infinite words.
\newblock {\em RAIRO Theor. Inform. Appl.}, 43(1):165--178, 2009.

\bibitem{Cusick_and_Flahive}
T.~W. Cusick and M.~E. Flahive.
\newblock The Markoff and Lagrange Spectra.
\newblock Math. Surveys and Monographs, no. 30, Amer. Math. Soc., Providence, Rhode Island, 1989.

\bibitem{Damanik200223}
D.~Damanik and D.~Lenz.
\newblock The index of {S}turmian sequences.
\newblock {\em European J. Combin.}, 23(1):23--29, 2002.

\bibitem{DR2012}
M.~Domaratzki and N.~Rampersad.
\newblock Abelian primitive words.
\newblock {\em Int. J. Found. Comput. Sci.}, 23(5):1021--1034, 2012.

\bibitem{Erdos1961221}
P.~Erd\H{o}s.
\newblock Some unsolved problems.
\newblock {\em Michigan Math. J.}, 4(3):291--300, 1957.

\bibitem{FiLaLeLeMiPG13}
G.~Fici, A.~Langiu, T.~Lecroq, A.~Lefebvre, F.~Mignosi, and \'E.~Prieur-Gaston.
\newblock Abelian {R}epetitions in {S}turmian {W}ords.
\newblock In {\em Proceedings of the 17th International Conference on
  Developments in Language Theory}, volume 7907 of {\em Lecture Notes in
  Computer Science}, pages 227--238. Springer Berlin Heidelberg, 2013.

\bibitem{Freiman}
G.~A.~Freiman.
\newblock  Diofantovy priblizheniya i geometriya chisel (zadacha Markova)
[Diophantine approximation and geometry of numbers (the Markov spectrum)].
\newblock {\em  Kalininskii Gosudarstvennyi Universitet}, Kalinin, 1975.

\bibitem{Glen200845}
A.~Glen, J.~Justin, and G.~Pirillo.
\newblock Characterizations of finite and infinite episturmian words via
  lexicographic orderings.
\newblock {\em European J. Combin.}, 29(1):45--58, 2008.

\bibitem{Hall}
M.~Hall. 
\newblock  On the sum and products of continued fractions.
\newblock {\em Annals of Math.}, 48:966--993, 1947.

\bibitem{Hardy_and_Wright}
G.~H. Hardy and E.~M. Wright.
\newblock {\em An Introduction to the Theory of Numbers}.
\newblock Clarendon Press, Oxford, 1979.
\newblock 5th edition.

\bibitem{JeZa04}
O.~Jenkinson and L.~Q. Zamboni.
\newblock Characterisations of balanced words via orderings.
\newblock {\em Theoret. Comput. Sci.}, 310(1-3):247--271, 2004.

\bibitem{Justin2001363}
J.~Justin and G.~Pirillo.
\newblock Fractional powers in {S}turmian words.
\newblock {\em Theoret. Comput. Sci.}, 255(1--2):363--376, 2001.

\bibitem{KnuthAMM}
D.~Knuth.
\newblock Sequences with precisely $k+1$ $k$-blocks, solution of problem {E}2307.
\newblock {\em Amer. Math. Monthly}, 79:773--774, 1972.

\bibitem{Krieger200770}
D.~Krieger.
\newblock On critical exponents in fixed points of non-erasing morphisms.
\newblock {\em Theoret. Comput. Sci.}, 376(1--2):70--88, 2007.

\bibitem{Lang}
S.~Lang.
\newblock {\em Introduction to Diophantine Approximations}.
\newblock Springer, 1995.

\bibitem{lothaire-book:2002}
M.~Lothaire.
\newblock {\em Algebraic Combinatorics on {W}ords}.
\newblock Cambridge University Press, Cambridge, U.K., 2002.

\bibitem{Ma}
A.~Markov.
\newblock Sur les formes quadratiques binaires ind\'{e}finies. 
\newblock {\em Math. Ann.}, 15:381--406, 1879.

\bibitem{Mi89}
F.~Mignosi.
\newblock Infinite {W}ords with {L}inear {S}ubword {C}omplexity.
\newblock {\em Theoret. Comput. Sci.}, 65(2):221--242, 1989.

\bibitem{MignosiPirillo}
F.~Mignosi and G.~Pirillo.
\newblock Repetitions in the {F}ibonacci infinite word.
\newblock {\em RAIRO Theor. Inform. Appl.}, 26:199--204, 1992.

\bibitem{MoHe38}
M.~Morse and G.~A. Hedlund.
\newblock Symbolic dynamics.
\newblock {\em Amer. J. Math.}, 60:1--42, 1938.

\bibitem{Parikh:1966:CL:321356.321364}
R.~J. Parikh.
\newblock On context-free languages.
\newblock {\em J. Assoc. Comput. Mach.}, 13(4):570--581, 1966.

\bibitem{Pelto}
J.~Peltom{\"a}ki.
\newblock Characterization of repetitions in {S}turmian words: {A} new proof.
\newblock {\em Inform. Proc. Lett.}, 115(11):886--891, 2015.

\bibitem{Perrin2012265}
D.~Perrin and A.~Restivo.
\newblock A note on {S}turmian words.
\newblock {\em Theoret. Comput. Sci.}, 429:265--272, 2012.

\bibitem{PZ13}
S.~Puzynina and L.~Q. Zamboni.
\newblock Abelian returns in {S}turmian words.
\newblock {\em J. Comb. Theory, Ser. A}, 120(2):390--408, 2013.

\bibitem{Pitheasfogg}
N.~Pytheas~Fogg.
\newblock {\em Substitutions in Dynamics, Arithmetics and Combinatorics},
  volume 1794 of {\em Lecture Notes in Math.}
\newblock Springer, 2002.

\bibitem{Richomme201179}
G.~Richomme, K.~Saari, and L.~Zamboni.
\newblock Abelian complexity of minimal subshifts.
\newblock {\em J. Lond. Math. Soc.}, 83(1):79--95, 2011.

\bibitem{Rigo13}
M.~Rigo, P.~Salimov, and E.~Vandomme.
\newblock Some properties of abelian return words.
\newblock {\em J. Integer Seq.}, 16:13.2.5, 2013.

\bibitem{SS2011}
A.~Samsonov and A.~Shur.
\newblock On {A}belian repetition threshold.
\newblock {\em RAIRO Theor. Inform. Appl.}, 46:147--163, 2012.

\bibitem{Vandeth2000283}
D.~Vandeth.
\newblock Sturmian words and words with a critical exponent.
\newblock {\em Theoret. Comput. Sci.}, 242(1--2):283--300, 2000.

\bibitem{yasutomi}
S.-I. Yasutomi.
\newblock On Sturmian sequences which are invariant under some substitutions.
\newblock In {\em Number theory and its applications} (Kyoto, 1997), pages 347--373. Kluwer Acad. Publ., Dordrecht, 1999.

\end{thebibliography}
\end{document}